\documentclass[reqno]{amsart}

\usepackage{amssymb}
\usepackage{amscd}
\usepackage{color}
\usepackage{amsfonts}
\usepackage{latexsym}
\newcommand{\ec}{\color{black}}

  \RequirePackage{amsbsy}
\newcommand{\co}{\boldsymbol{c}}

\newcommand{\QInv}{\text{\rm QInv}}
\newcommand{\Cl}{\text{\rm Cl}}
\newcommand{\G}{\text{\rm G}}
\newcommand{\Pic}{\text{\rm Pic}}
\newcommand{\Prin}{\text{\rm Prin}}

\newcommand{\Inv}{\text{\rm Inv}}
\newcommand{\calT}{\boldsymbol{\mathcal{T}}}

\newcommand{\f}{\boldsymbol{f}}
\newcommand{\F}{\boldsymbol{F}}
  \ec
\newtheorem{theorem}{Theorem}[section]
\newtheorem{lemma}[theorem]{Lemma}
\newtheorem{corollary}[theorem]{Corollary}

\newtheorem{example}[theorem]{Example}
\newtheorem{remark}[theorem]{Remark}
\newtheorem{proposition}[theorem]{Proposition}

\title[P$\star$MD's  and class groups]{Some
remarks  on Pr\"{u}fer $\star $--multiplication domains
 and class groups}
\author[D.F. Anderson, M. Fontana, and M. Zafrullah]{David Anderson, Marco Fontana 
 and Muhammad Zafrullah}
\address{D.A.: Department of Mathematics, University of
Tennessee at Knoxville, Knoxville, TN 37996-1300.}
\email{anderson@math.utk.edu}
\address{M.F.: \ Dipartimento di Matematica, Universit\`a degli Studi
``Roma Tre'', Rome, Italy }
\email{fontana@mat.uniroma3.it }
\address{M.Z.: \ 57 Colgate Street, Pocatello, ID 83201, USA}
\email{mzafrullah@usa.net}
\thanks{\it Acknowledgments. \rm During the preparation of this paper, the second named author
was partially supported by  a grant PRIN-MiUR. \\
$\mbox {\ \ \ }$ A reference to \cite{[FJS]} by Said El-Baghdadi turned out to be the starting point of this paper. We are grateful to him. \  We also would like to thank the referee for several helpful suggestions.}

\date{\today }

\begin{document}

\begin{abstract}
  Let $D$ be an integral domain with quotient field $K$ and let $X$ be an
indeterminate over $D$. Also, let   $\boldsymbol{\mathcal{T}}:=\{T_{\lambda }\mid  \lambda
 \in \Lambda \}$ be a
defining family of quotient rings of $D$ and suppose that $\ast $ is a
finite type star operation on $D$ induced by $\boldsymbol{\mathcal{T}}$.  We show that $D$ is a P$
\ast $MD (resp., P$v$MD) if and only if $(\co_D(fg))^{\ast }=(\co_D(f)\co_D(g))^{\ast }$
(resp., $(\co_D(fg))^{w }=(\co_D(f)\co_D(g))^{w }$) for all $0 \ne f,g \in K[X]$.  A more general version of this result is given in the semistar operation  setting. We
give a method for recognizing P$v$MD's which are not P$\ast $MD's for a
certain finite type star operation $\ast $.  We   study  domains $D$
for which the $\ast $--class group $\Cl^{\ast }(D)$ equals the $t$--class
group $\Cl^{t}(D)$ for any finite   type   star operation $\ast $, and we
indicate examples of P$v$MD's $D$ such that $\Cl^{\ast }(D)\subsetneq
\Cl^{t}(D)$.   We also compute  $\Cl^v(D)$  for certain
valuation domains  $D$.  
\end{abstract}

\maketitle


\section*{Introduction and Background}

Let $D$ be an integral domain with quotient field $K$. Let $
\boldsymbol{\overline{F}}(D)$ be the set of all nonzero
$D$--submodules of $K$ and let $\boldsymbol{F}(D)$ be the set of
all nonzero fractional ideals of $D$, i.e., $E \in
\boldsymbol{F}(D)$ if $E \in \boldsymbol{ \overline{F}}(D)$ and
there exists a $0 \ne d \in D$ with $dE \subseteq D$. Let
$\boldsymbol{f}(D)$ be the set of all nonzero finitely generated
$D$--submodules of $K$. Then, obviously $\boldsymbol{f}(D)
\subseteq \boldsymbol{F}(D) \subseteq
\boldsymbol{\overline{F}}(D)$.

A \emph{semistar operation} on $D$ is a map $\star: \boldsymbol{\overline{F}}
(D) \to \boldsymbol{\overline{F}}(D), E \mapsto E^\star$, such that the
following properties hold for all $0 \ne x \in K$ and all $E,F \in
\boldsymbol{\overline{F}}(D)$:

\begin{enumerate}
\item[$(\star_1)$] $(xE)^\star=xE^\star$;

\item[$(\star_2)$] $E \subseteq F$ implies $E^\star \subseteq F^\star$;

\item[$(\star_3)$] $E \subseteq E^\star$ and $E^{\star \star}   :=  \left(E^\star
\right)^\star=E^\star$.
\end{enumerate}

Given a semistar operation $\star$ on $D$, the following basic formulas,
which hold for all $E,F \in \boldsymbol{\overline{F}}(D)$, follow easily
from the axioms:
\begin{equation*}
\begin{array}{rl}
(EF)^\star = & \hskip -7pt (E^\star F)^\star =\left(EF^\star\right)^\star
=\left(E^\star F^\star\right)^\star\,; \\
(E+F)^\star = & \hskip -7pt \left(E^\star + F\right)^\star= \left(E +
F^\star\right)^\star= \left(E^\star + F^\star\right)^\star\,; \\
(E:F)^\star \subseteq & \hskip -7pt (E^\star :F^\star) = (E^\star :F)  =
\left(E^\star :F\right)^\star,\;\, \mbox{\rm if \ } (E:F) \neq (0)\,; \\
(E\cap F)^\star \subseteq & \hskip -7pt E^\star \cap F^\star = \left(E^\star
\cap F^\star \right)^\star. 
\end{array}
\end{equation*}
\noindent (Cf. for instance \cite[Theorem 1.2 and p. 174]{[FH2000]}.)

A \emph{(semi)star operation} is a semistar operation which when restricted
to $\boldsymbol{F}(D)$ is a star operation (the reader may consult    \cite[Sections
32 and 34]{[G]}   for a quick review of star
operations,  which are denoted by the symbol $\ast$).   It is easy to see that a semistar operation $\star$ on $D$ is a
(semi)star operation if and only if $D^\star = D$.

{ If $\star$ is a semistar operation on $D$, then there is a map\
$\star_{\!_f}: \boldsymbol{\overline{F}}(D) \to \boldsymbol{
\overline{F}}(D)$ defined as follows:
\begin{equation*}
E^{\star_{\!_f}}:=\bigcup \{F^\star\mid \ F \in \boldsymbol{f}(D)
\mbox{ and } F \subseteq E\} \;\; \;\; \mbox{  for all } E \in
\boldsymbol{\overline{F}}(D)  .
\end{equation*}
\noindent It is easy to see that $\star_{\!_f}$ is a semistar
operation on $ D $, called the \emph{semistar operation of finite
type associated to $\star$}. Note that $F^\star=F^{\star_{\!_f}}$
for all $F \in \boldsymbol{f}(D)$. A semistar operation $\star$ is
called a \emph{semistar operation of finite type} (or a
\emph{semistar operation of finite character}) if $
\star=\star_{\!_f}$. It is easy to see that
$(\star_{\!_f}\!)_{\!_f}=\star_{ \!_f}$ (i.e., $\star_{\!_f}$ is
of finite type). }

\medskip

If $\star_1$ and $\star_2$ are two semistar operations on $D$, we
say that $ \star_1 \leq \star_2$ \ if $E^{\star_1} \subseteq
E^{\star_2}$ for all $E \in \boldsymbol{\overline{F}}(D)$. This is
equivalent to saying that $ \left(E^{\star_{1}}\right)^{\star_{2}}
 =  E^{\star_2}  =
\left(E^{\star_{2}}\right)^{\star_{1}}$ for all $E \in
\boldsymbol{\overline{F}}(D)$. Obviously, for  any   semistar
operation $\star$ on $D$, we have $\star_{\!_f} \leq \star$, and
if $\star_1 \leq \star_2$, then $ (\star_1)_{{\!_f}} \leq
(\star_2)_{{\!_f}}$.

\medskip

Let $I\subseteq D$ be a nonzero ideal of $D$ and let $\star $ be a
semistar operation on $D$. We say that $I$ is a
\textit{quasi--$\star $--ideal} (resp., \it $\star $--ideal\rm) of $D$ if
$I^{\star }\cap D=I$ (resp., $I^{\star }=I $). \ Similarly, we
call a quasi--$\star $--ideal (resp.,  $\star $--ideal) of $D$ a
\it  quasi--$\star $--prime \rm  (resp., \it $\star $--prime\rm) \it
 ideal  \rm  of $D$ if it is also a prime ideal. We call a
maximal element in the set of all proper quasi--$\star $--ideals
(resp., $\star $--ideals) of $D$ a  \it  quasi--$\star $--maximal \rm  (resp., \it  $\star $--maximal\rm)  \it ideal  \rm
of $D$. Note that if $I\subseteq D$ is a $\star $--ideal, then it is also
a quasi--$ \star $--ideal, and when $\,D=D^{\star }$ (i.e., when
$\star $ is a(semi)star operation), the notions of quasi--$\star
$--ideal and $\star $--ideal coincide.

It is not hard to prove that a quasi--$\star $--maximal ideal is a
prime ideal and that each proper quasi--$\star _{_{\!f}}$--ideal
is contained in a quasi--$\star _{_{\!f}}$--maximal ideal. More
details can be found in \cite[page 4781]{[FL03]}.
We will denote the set of quasi--$\star $--prime (resp., $\star
$--prime) ideals of $D$ by $\text{QSpec}^{\star }(D)$ (resp.,
$\text{Spec}^{\star }(D)$) and the set of quasi--$\star $--maximal
(resp., $\star $--maximal) ideals of $D$ by $\text{QMax}^{\star
}(D)$ (resp., $\text{Max}^{\star }(D)$). By the previous
observations, we have that $\text{QMax}^{\star _{_{\!f}}}(D)$
(resp., $\text{Max}^{\star _{_{\!f}}}(D)$) is non-empty for each
semistar (resp., (semi)star) operation $\star $ on $D$.

\medskip

If $T$ is an overring of $D$, then we can define a semistar
operation $\star _{\{T\}}$  on $D$  by $E^{\star _{\{T\}}}:=ET$
for all $E\in \boldsymbol{\overline{F}}(D)$. It is easily seen
that $\star _{\{T\}}$ is a semistar (non (semi)star, if
$D\subsetneq T$) operation of finite type  on $D$.

If $\{\star _{\lambda }\mid \lambda \in \Lambda \}$ is a family of semistar
operations on $D$, then $\wedge \{\star _{\lambda }\mid \lambda \in \Lambda
\}$ is the semistar operation on $D$ defined as follows:
\begin{equation*}
E^{\wedge \{\star _{\lambda }\mid \lambda \in \Lambda \}}:=\bigcap
\{E^{\star _{\lambda }}\mid \lambda \in \Lambda \}   \mbox{ \;  for
all } E \in \boldsymbol{\overline{F}}(D).
\end{equation*}
In particular, if $\boldsymbol{\mathcal{T}}:=\{T_{\lambda }\mid  {\lambda
}  \in \Lambda \}$ is a given family of overrings of $D$, then $\wedge _{
\boldsymbol{\mathcal{T}}}$ denotes the semistar operation $\wedge \{\star
_{\{T_{\lambda }\}}\mid \lambda \in \Lambda \}$.

Let $\Delta $ be a set of prime ideals of an integral domain $D$
and $ \boldsymbol{\mathcal{L}}(\Delta )$ the set of localizations
$\{D_{P}\mid P\in \Delta \}$. The semistar operation $\star
_{\Delta }:=\wedge _{ \boldsymbol{\mathcal{L}}(\Delta )}$ is
called the  \it  spectral semistar operation associated to
$\Delta $. \rm   A semistar operation $\star $ on an integral
domain $D$ is called a \it  spectral semistar operation  \rm
if there exists a subset $\Delta $ of the prime spectrum
$\text{Spec}(D)$ of $D$ such that $ \star  =   \star
_{\Delta }$. Note that   for  $\Delta =\emptyset $, we set
$\star _{\Delta }:=\star _{\{K\}}$, where $K$ is the quotient
field of $D$.

When $\Delta :=\Delta (\star _{_{\!f}}):=\text{QMax}^{\star _{_{\!f}}}(D)$,
we set $\widetilde{\star }:=\star _{\Delta (\star _{_{\!f}})}$, i.e.,
\begin{equation*}
E^{\widetilde{\star }}:=\bigcap \left\{ ED_{M}\,|\,M\in \text{QMax}^{\star
_{_{\!f}}}(D)\right\}\;\;\;\;  \mbox{\rm  for all $E \in
\boldsymbol{\overline{F}}(D)$}.
\end{equation*}
\vskip 8pt

A semistar operation $\star$ is said to be \emph{stable} if $
(E \cap F)^\star = E^\star \cap F^\star$ for all $E,F \in \boldsymbol{\overline{F}}(D)$.
 Note that if $T$ is an overring of an integral domain $D$, then $
\star_{\{T\}}$ is stable if and only if $T$ is $D$--flat (cf.
\cite[Proposition 1.7]{Uda} and \cite[Theorem 7.4(i)]{matsumura}).
Clearly, if $ \{\star_{\lambda} \mid \lambda \in \Lambda \}$ is a
family of stable semistar operations on $D$, then $\wedge
\{\star_{\lambda} \mid \lambda \in \Lambda \}$ is also a stable
semistar operation on $D$. In particular, if $
\boldsymbol{\mathcal{T}}$ is a family of flat overrings of $D$,
then $ \wedge_{\boldsymbol{\mathcal{T}}}$ is  a stable semistar
operation on $D$.  Thus  every spectral semistar operation is
stable (cf. also \cite[Lemma 4.1(3)]{[FH2000]}).

It is  well known  that the semistar operation $\widetilde{\star
}$ is a stable semistar operation of finite type \cite[Corollaries
3.9 and 4.6]{[FH2000]}. We call $\widetilde{\star }$ the
\textit{stable semistar operation of finite type associated to
$\star $. }Furthermore, it is not hard to prove that $
\text{QMax}^{\widetilde{\star }}(D)=\text{QMax}^{\star
_{_{\!f}}}(D)$ \cite[Corollary 3.5(2)]{[FL03]}; thus
$\widetilde{\widetilde{\star }}=\widetilde{ \star }=\widetilde{\
\star _{_{\!f}}}$. Clearly $\widetilde{\star }\leq \star $, and
since $(\widetilde{\star })_{_{\!f}}=\widetilde{\star }$, then $
\widetilde{\star }\leq \star _{_{\!f}}\leq \star $. Moreover, it
is known that if $\star _{1}\leq \star _{2}$, then $\widetilde{\
\star _{1}}\leq \widetilde{\ \star _{2}}$ \cite[Propositions 3.1
and 3.4(3)]{[FL03]}.

\medskip

For each $E\in \boldsymbol{\overline{F}}(D)$, set $E^{-1}  :=
 (D:E):=\{z\in K\mid zE\subseteq D\}$. Clearly  $E\in
\boldsymbol{\overline{F}} (D)\!\setminus \!\boldsymbol{{F}}(D)$ if
and only if $E^{-1}=\{0\}$. As usual, we let  $ v_{D}$ (or just
$v$) denote the $v$--(semi)star operation defined by
$E^{v}:=(D:(D:E))=\left( E^{-1}\right) ^{-1}$ for all $E\in
\boldsymbol{\overline{F}}(D)$. (Note that $E\in
\boldsymbol{\overline{F}} (D)\!\setminus \!\boldsymbol{{F}}(D)$
implies that $E^{v}=K$.)\ We denote $(v_{D})_{_{\!f}}$ by $t_{D}$
(or just by $t$), the $t$--(semi)star operation on $D$; and we
denote the stable semistar operation of finite type associated to
$v_{D}$ (or, equivalently, to $t_{D}$) by $w_{D}$ (or just by
$w$), i.e., $ w_{D}:=\widetilde{v_{D}}=\widetilde{t_{D}}$. Clearly
$w_{D}\leq t_{D}\leq v_{D}$. Moreover, from \cite[Theorem
34.1(4)]{[G]}, we immediately deduce that $\star \leq v_{D}$, and
thus $\widetilde{\star }\leq w_{D}$ and $\star _{_{\!f}}\leq
t_{D}$, for each (semi)star operation $\star $ on $D$.

\medskip

\noindent \bf Remark. \rm Note that the (semi)\-star
operation $\,\widetilde{v}\,$ coincides with the (semi)star operation
defined as follows:
\begin{equation*}
E^{w}:=\bigcup \{(E:H)\;|\;\,H\in \boldsymbol{f}(D)\mbox{ and
}H^{v}=D\}\,  \;\;\mbox{ for all }E\in
\boldsymbol{\overline{F}}(D)\,.
\end{equation*}
In the \textquotedblleft star operation setting\textquotedblright, this operation was first considered by J. Hedstrom and E.
Houston in 1980 \cite[ Section 3]{Hedstrom/Houston: 1980} under
the name of the F$_{\infty }$--operation. Later, from 1997, this
operation was intensively studied by F. Wang and R. McCasland (cf.
\cite{[FM]} and \cite{Fanggui/McCasland:1999}) under the name of
the \it  $w$--operation. \rm Also note that the notion of $w$--ideal
coincides with the notion of semi-divisorial ideal considered by
S. Glaz and W. Vasconcelos in 1977 \cite{Glaz/Vasconcelos:1977}.
Finally, in 2000, for each star operation $\ast $  on $D$,
D.D. Anderson and S.J. Cook \cite {[AC]} considered the star
operation $\ast_{w}$  on $D$   defined as follows:
\begin{equation*}
E^{\ast_{w}}:=\bigcup \{(E:H)\;|\;\,H\in \boldsymbol{f}(D)\mbox{
and } H^{\ast }=D\}\,  \;\;\mbox{ for all }E\in
\boldsymbol{{F}}(D)\,.
\end{equation*}
It can be shown that when  $\star = \ast $   is a star operation, then
 $\ast _{w}$  coincides with $\widetilde{\ast}$  (defined in the
obvious way as a star operation on $\boldsymbol{{F}}(D)$)
\cite[Corollary 2.10]{[AC]}.

\mbox{}  Finally, note that a deep link between the semistar
operations of type $\widetilde{\star }\,$ and localizing systems of ideals
was established by M. Fontana and J. Huckaba in \cite{[FH2000]}.

\bigskip

Let $\star $ be a semistar operation on  the   integral domain $D$.

For $I\in \boldsymbol{\overline{F}}(D)$, we say that $I$ is  \it
$\star $--finite  \rm if there exists a $J\in \boldsymbol{f}(D)$
such that $J^{\star }=I^{\star }$. (Note that in the  above  
definition, we do not require that $J\subseteq I $.) It is
immediate to see that if $\star _{1}\leq \star _{2}$ are semistar
operations and $I$ is $\star _{1}$--finite, then $I$ is $\star
_{2}$ --finite. In particular, if $I$ is $\star
_{\!_{f}}$--finite, then it is $ \star $--finite. The converse is
not true in general, and one can prove that $I$ is $\star
_{\!_{f}}$--finite if and only if there exists
 $J\in \boldsymbol{f}(D)$, $J\subseteq I$, such that $J^{\star
}=I^{\star }$ \cite[Lemma 2.3]{FP}. This result was proved in the star
operation setting by M. Zafrullah in \cite[Theorem 1.1]{[Z1989]}.

For $I$  a nonzero ideal of $D$, we say that $I$ is \it{$\star
$--invertible} \rm if $(II^{-1})^{\star }=D^{\star }$. From the
fact that $\text{QMax}^{\widetilde{ \star }}(D)=\text{QMax}^{\star
_{\!_{f}}}(D)$, it easily follows that an ideal $I$ is
$\widetilde{\star }$--invertible if and only if $I$ is $\star
_{_{\!f}}$--invertible (note that if $\boldsymbol{\star} $ is a semistar
operation of finite type, then $(II^{-1})^{\boldsymbol{\star} }=D^{\boldsymbol{\star}  }$ if
and only if $ II^{-1}\not\subseteq M$ for all $M\in
\text{QMax}^{\boldsymbol{\star} }(D)$). It is   well known   that if $I$ is $\star
_{_{\!f}}$--invertible, then $I$ and $I^{-1}$ are both $\star
_{_{\!f}}$--finite \cite[Proposition 2.6]{FP}.

\medskip

An integral domain $D$ is called a \it  Pr\"{u}fer $\star
$--multiplication domain  \rm (for short,  \it P$\star $MD\rm) if
every nonzero finitely generated ideal of  $D$ is $\star
_{_{\!f}}$--invertible (cf. for instance \cite{[FJS]}). Note that 
    for $\star = \ast$   a   star operation of finite type on $D$,
P$\ast$MD's were intro\-duced by Houston, Malik, and Mott in
\cite{[HMM]} as $\ast$--multiplication domains (for short,
$\ast$--MD's).  When $\star = v$, we have the classical notion
of P$v$MD (cf. \!for in\-stance \cite{Gr}, \cite{MZ} and \cite{K89b});
when $\star =d$, where $d$ denotes the identity (semi)star
operation, we have the notion of Pr\"{u}fer domain \cite[Theorem
22.1]{[G]}. Note that from the definition and from the previous
observations, it immediately follows that the notions of P$\star
$MD, P$\star _{_{\!f}}$MD,  and P$ \widetilde{\star }$MD coincide.

\medskip

Let $K$ be the quotient field of an integral domain $D$ and let
$X$ be an indeterminate over $K$. For each $0\neq h\in K[X]$, we
denote by $ \boldsymbol{c}_{D}(h)$ the \it content of $h$ with
respect to $D$, \rm  i.e., the $D$--submodule of $K$ generated by
the coefficients of $h$. Clearly $\boldsymbol{c} _{D}(h)\in
\boldsymbol{f}(D)$, and if $T$ is an overring of $D$, then $
\boldsymbol{c}_{T}(h)=\boldsymbol{c}_{D}(h)T$.

Gauss' Lemma for the content of polynomials holds for Dedekind domains (or,
more generally, for Pr\"{u}fer domains).  A more precise
statement is the following:

\medskip
 \noindent  \bf  Gauss-Gilmer-Tsang Theorem  \rm
\cite[Corollary 28.5]{[G]}. \it Let $D$ be an integral domain with
quotient field $K$. Then $D$ is a Pr\"{u}fer domain if and only if
$\boldsymbol{c}
_{D}(fg)=\boldsymbol{c}_{D}(f)\boldsymbol{c}_{D}(g)$   for all $0
\ne f,g\in K[X]$. \rm
\medskip

\noindent \bf  Remark. 
 \rm   W. Krull \cite[page 557]{[Krull]}
showed that if $D$ is an integrally closed domain with quotient
field $K$, then  we have
$(\boldsymbol{c}_{D}(fg))^{v}=(\boldsymbol{c}_{D}(f)\boldsymbol{c}_{D}(g))^{v}$
 for all $0 \ne f,g\in K[X]$, and called it Gauss' Theorem.
Obviously the currently known Gauss' Lemma (that goes as: the
product of two primitive polynomials over   a UFD is again primitive \cite[page 165]{[Co]})   and  Gauss'  own statement  
(let $f$ and $g$ be monic polynomials in
one indeterminate with rational coefficients, if the coefficients of $f$  and 
$g$ are not all integers,  then the coefficients of $fg$ are not all integers
\cite[page 1]{[Ed]})  
follow from Krull's above-mentioned result. 
 As
pointed out before the statement of the above theorem, Krull's
Gauss' Theorem also holds for Pr\"{u}fer domains; because for $D$
a Pr\"{u}fer domain with quotient field $K$, we have
$(\boldsymbol{c}_{D}(f))^{v}=\boldsymbol{c}_{D}(f)$ for  all
  $ 0 \ne f\in K[X].$ Thus Krull's Gauss' Theorem for
Pr\"{u}fer domains becomes: if $D$ is a Pr\"{u}fer domain with
quotient field $K$, then
$\boldsymbol{c}_{D}(fg)=\boldsymbol{c}_{D}(f)
\boldsymbol{c}_{D}(g)$  for  all   $0 \ne f,g\in K[X]$. The
converse of this statement was included in H. Tsang's unpublished
dissertation \cite {[Tsang]}. This result was later, and
independently, rediscovered by R. Gilmer and published in
\cite{[Gi]}. Since neither of these authors attributed their
result to Gauss,  we feel it appropriate to include their names
with Gauss' name. 
For more on the history of Gauss' Lemma, the reader may
consult Anderson \cite{[A1]}.

\medskip

For general integral domains, we always have the inclusion of
ideals $ \boldsymbol{c}_{D}(fg)\subseteq
\boldsymbol{c}_{D}(f)\boldsymbol{c}_{D}(g)$, and more precisely we
have the following:

\medskip

\noindent  \bf Dedekind--Mertens Lemma  \rm  \cite[Theorem
28.1]{[G]}. \it  Let $ 0 \ne f,g \in K[X]$ and let $m:=\deg (g)$.
Then
\begin{equation*}
\boldsymbol{c}_{D}(f)^{m}\boldsymbol{c}_{D}(fg)=\boldsymbol{c}_{D}(f)^{m+1}
\boldsymbol{c}_{D}(g)\,.
\end{equation*} \rm
\smallskip

In Section 1, we prove a semistar extension of  the
Gauss-Gilmer-Tsang Theorem  (as stated above), i.e., we show that if $\star $
is a stable semistar operation  of finite type
defined on an integral domain $D$,  then $D$ is a P$\star $MD if
and only if
$\boldsymbol{c}_{D}(fg)^{\star}=(\boldsymbol{c}_{D}(f)\boldsymbol{c}_{D}(g))^{\star
}$    for all   $ 0 \ne f,g\in K[X]$.  Using this result, we show
that there is an abundance of P$v$MD's which are not P$\star $MD's
for appropriate stable (semi)star operations $\star $ of finite
type on $D$.

 For
a finite type star operation $\ast $ on $D$, let $\Inv^\ast(D)$ be the
group of $\ast $--invertible $\ast $--ideals of $D$ under $\ast
$--multiplication and let $ \Prin(D)$ be the subgroup of nonzero
principal fractional ideals of $D$. Call $ \Cl^{\ast }(D):=
\Inv^\ast(D)/\Prin(D)$ the $\ast $--class group of  $D$. The $\ast
$--class groups were discussed in \cite{[An]}. 

 In Section 2, we study the $\ast $--class group and identify a situation in
which for every finite type star operation $\ast $ on  $D$, we
have $\Cl^{\ast }(D)=\Cl^{t}(D)$; and using the results of Section
1,  we give examples of integral domains   $D$  for which $\Cl^{\ast
}(D)\subsetneq \Cl^{t}(D)$ for some finite type star operation
$\ast$ on $D$. 

In Section 3, we deepen the study of the $v$--class group  with  special    attention to the case of valuation domains.  In particular, we compute
$\Cl^v(D)$  when  $D$  is a valuation domain with branched maximal ideal.

\section{Pr\"{u}fer $\star $--multiplication domains}

 With all the introduction at hand, we start right away with the
promised characterization of P$\star $MD's. Using Theorem \ref{Theorem A'}
below, we conclude that $D$ is a P$v$MD if and only if $\boldsymbol{c}
_{D}(fg)^{w}=(\boldsymbol{c}_{D}(f)\boldsymbol{c}_{D}(g))^{w}$ for all $
0\neq f,g\in K[X]$. Also, using the proof of Theorem \ref{Theorem A'}, we
give a method for recognizing a P$v$MD which has a stable (semi)star
operation $\star $ of finite type such that $D$ is not a P$\star $MD.

\begin{theorem}
\label{Theorem A'} Let $D$ be an integral domain with quotient field $K$,
let $X$ be an indeterminate over $K$, and let $\star $ be a stable semistar
operation of finite type defined on $D$. Then the following are equivalent:

\begin{enumerate}
\item[(i)] $\boldsymbol{c}_{D}(fg)^{\star }=(\boldsymbol{c}_{D}(f)
\boldsymbol{c}_{D}(g))^{\star }$ for all $0\neq f,g\in K[X]$.

\item[(ii)]  $D_{M}$ is a valuation domain for all $M\in \text{\rm QMax}^{\star
}(D)$.

\item[(iii)] $D$ is a P$\star $MD.
\end{enumerate}
\end{theorem}

\begin{proof}
By the observations in the previous section, we know that under
the present hypotheses, $\star =\widetilde{\star }$ \cite[Corollary 3.9(2)]{[FH2000]}, and thus
$F^{\star }D_{M}=FD_{M}$ for all $M\in \text{QMax}^{\star }(D)$
and $F\in \boldsymbol{f}(D)$.

(i)$\Rightarrow $(ii) Let $M\in \text{QMax}^{\star }(D)$. From
(i), we deduce that
$\boldsymbol{c}_{D_{M}}(fg)=\boldsymbol{c}_{D}(fg)D_{M}=
\boldsymbol{c}_{D}(fg)^{\star
}D_{M}=(\boldsymbol{c}_{D}(f)\boldsymbol{c} _{D}(g))^{\star
}D_{M}=\boldsymbol{c}_{D}(f)\boldsymbol{c}_{D}(g)D_{M}=
\boldsymbol{c}_{D_{M}}(f)\boldsymbol{c}_{D_{M}}(g)$. 
This implies that $D_{M}$ is a valuation domain (i.e., a local
Pr\"{u}fer domain) by the Gauss-Gilmer-Tsang Theorem.

(ii)$\Rightarrow $(iii) Let $F\in \boldsymbol{f}(D)$. Note that
for each flat overring $T$ of $D$, we have $F^{-1}T=(FT)^{-1}$.
 Also recall   that  for  all   $M\in
\text{QMax}^{\star }(D)$,   every  nonzero finitely generated  
ideal is invertible in the valuation domain $D_{M}$.
Therefore, we have that $ (FF^{-1})^{\star }=\bigcap
\{(FF^{-1})D_{M}\mid M\in \text{QMax}^{\star }(D)\}=\bigcap
\{FD_{M}F^{-1}D_{M}\mid M\in \text{QMax}^{\star }(D)\}=\bigcap
\{(FD_{M}(FD_{M})^{-1}\mid M\in \text{QMax}^{\star }(D)\}=\bigcap
\{D_{M}\mid M\in \text{QMax}^{\star }(D)\}=D^{\star }$.

(iii)$\Rightarrow $(i) By the Dedekind-Mertens Lemma, $\boldsymbol{c}
_{D}(f)^{m}\boldsymbol{c}_{D}(fg)=\boldsymbol{c}_{D}(f)^{m+1}\boldsymbol{c}
_{D}(g)$ for all $0\neq f,g\in K[X]$, where $m= \deg (g)$. In particular, we
have $(\boldsymbol{c}_{D}(f)^{m}\boldsymbol{c}_{D}(fg))^{\star }$ = $(
\boldsymbol{c}_{D}(f)^{m+1}\boldsymbol{c}_{D}(g))^{\star }$. Since $D$ is a P$\star $MD, if $F:=\boldsymbol{c}_{D}(f)\in \boldsymbol{f}(D)$, then $
(FF^{-1})^{\star }=D^{\star }=(F^{m}(F^{m})^{-1})^{\star }$. Therefore:
\begin{equation*}
\begin{array}{rl}
\boldsymbol{c}_{D}(fg)^{\star }= &\hskip -5pt ((F^{m}(F^{m})^{-1})^{\star }
\boldsymbol{c}_{D}(fg))^{\star } \\
= &\hskip -5pt (\boldsymbol{c}_{D}(f)^{m}(F^{m})^{-1}\boldsymbol{c}_{D}(fg))^{\star }=(
\boldsymbol{c}_{D}(f)^{m+1}(F^{m})^{-1}\boldsymbol{c}_{D}(g))^{\star } \\
= &\hskip -5pt
(F^{m}(F^{m})^{-1}\boldsymbol{c}_{D}(f)\boldsymbol{c}_{D}(g))^{\star
} = ((F^{m}(F^{m})^{-1})^{\star
}\boldsymbol{c}_{D}(f)\boldsymbol{c}
_{D}(g))^{\star } \\
= &\hskip -5pt (\boldsymbol{c}_{D}(f)\boldsymbol{c}_{D}(g))^{\star }.
\end{array}
\end{equation*}
\vskip -15pt \end{proof}

\begin{corollary}
\label{A1} Let $D$ be an integral domain with quotient field $K$, let $X$ be
an indeterminate over $K$, and let $\star $ be a semistar operation defined
on $D$. Then the following are equivalent:

\begin{enumerate}
\item[(i)] $\boldsymbol{c}_{D}(fg)^{\widetilde{\star }}=(\boldsymbol{c}
_{D}(f)\boldsymbol{c}_{D}(g))^{\widetilde{\star }}$ for all $ 0 \neq   f,g\in
K[X]$.

\item[(ii)]  $D_{M}$ is a valuation domain for all $M\in \text{\rm QMax}^{\star
_{_{\!f}}}(D)$.

\item[(iii)] $D$ is a P$\star $MD.
\end{enumerate}
\end{corollary}

\begin{proof} Apply Theorem \ref{Theorem A'} to $\widetilde{\star }$, the stable
semistar operation of finite type associated to $\star $. Recall that  from $
\text{QMax}^{\star _{_{\!f}}}(D)=\text{QMax}^{\widetilde{\star }}(D)$, we
already deduced that the notions of P$\star $MD and P$\widetilde{\star}$MD coincide. \end{proof}

\begin{corollary}
\label{Corollary A2} Let $D$ be an integral domain and let $\star
$ be a semistar operation of finite type induced by a family
$\boldsymbol{ \mathcal{T}}$ of flat overrings of $D$, i.e., $\star
=\wedge _{ \boldsymbol{\mathcal{T}}}$. Then $D$ is a P$\star $MD
if and only if $\boldsymbol{c}_{T}(fg)=
\boldsymbol{c}_{T}(f)\boldsymbol{c}_{T}(g)$  for all  $0 \ne
f,g\in K[X]$ (i.e., $T$ is a Pr\"{u}fer domain) and all $T\in
\boldsymbol{\mathcal{T}}$.
\end{corollary}

\begin{proof}
Note that in this case, $\star $ is stable because each overring
$T\in \boldsymbol{\mathcal{T}}$ is flat, and $\star $ is of finite
type by assumption. Therefore $\star =\widetilde{\star }$.
Moreover, we have $ \boldsymbol{c}_{D}(h)^{\star
}T=\boldsymbol{c}_{D}(h)T=\boldsymbol{c}_{T}(h)$ for all $0\neq
h\in K[X]$. The conclusion then follows immediately from Theorem
\ref{Theorem A'} since $\boldsymbol{c}_{D}(fg)^{\star }=\bigcap \{
\boldsymbol{c}_{T}(fg)\mid T\in \boldsymbol{\mathcal{T}}\}$ and $(
\boldsymbol{c}_{D}(f)\boldsymbol{c}_{D}(g))^{\star }=\bigcap
\{\boldsymbol{c} _{T}(f)\boldsymbol{c}_{T}(g)\mid T\in
\boldsymbol{\mathcal{T}}\}$. \end{proof}

\begin{remark} \rm
Let $D$ be an integral domain, $\Delta $ a subset of
$\text{Spec}(D)$, and $ \star :=\star _{\Delta }$, the spectral
semistar operation associated to $ \Delta $. If we assume that
$\Delta $ is quasi-compact (as a subspace of $ \text{Spec}(D)$
endowed with the Zariski topology), then the semistar operation
$\star $ is a stable semistar operation of finite type   \cite[Corollary 4.6(2)]{[FH2000]},  
and thus we can apply Corollary
\ref{Corollary A2} to this case.
\end{remark}

The next corollary is a particularly significant case of Corollary
\ref{Corollary A2}.

\begin{corollary}
\label{Corollary A3} Let $D$ be an integral domain with quotient
field $K$, let $X$ be an indeterminate over $K$, and let $\star $
be the (semi)star operation of finite type induced by a defining
family $\boldsymbol{ \mathcal{T}}$ of $D$ consisting of quotient
rings of $D$, i.e., $ \boldsymbol{\mathcal{T}}:=\{T_{\lambda }\mid
\lambda \in \Lambda \}$ with $ D=\bigcap \{T_{\lambda }\mid
\lambda \in \Lambda \}$ and each $T_{\lambda }$ is a ring of
fractions of $D$. Then the following are equivalent:

\begin{enumerate}
\item[(i)] $\boldsymbol{c}_{D}(fg)^{\star }=(\boldsymbol{c}_{D}(f)
\boldsymbol{c}_{D}(g))^{\star }$ for all $0\neq f,g\in K[X]$.

\item[(ii)] Each $T_{\lambda }\in \boldsymbol{\mathcal{T}}$ is a Pr\"{u}fer
domain.

\item[(iii)] $D$ is a P$\star $MD.
\end{enumerate}
\end{corollary}

\smallskip

Since the $w$--operation is the (semi)star operation on $D$ induced by the
quotient rings $\boldsymbol{\mathcal{T}}:=\{D_{Q}\mid Q\in \text{Max}
^{t}(D)\}$, i.e., $w=\wedge _{\boldsymbol{\mathcal{T}}}$, and since $w$ is
of finite type, we have the following application of the previous corollary.

\begin{corollary}
\label{Corollary B'} An integral domain $D$ is a P$v$MD if and only if $
\boldsymbol{c}_{D}(fg)^{w}=(\boldsymbol{c}_{D}(f)\boldsymbol{c}_{D}(g))^{w}$
for all $0\neq f,g\in K[X]$.
\end{corollary}

Proof. Apply Corollary \ref{Corollary A3} and recall that, as a consequence
of the fact that P$\widetilde{\star }$MD = P$\star $MD, we have P$w$MD = P$v$MD. \hfill $\Box $

\bigskip

This corollary on the one hand gives a nice general
characterization of P$v$MD's, and on the other hand it establishes
the \textquotedblleft superiority\textquotedblright\ of the
$w$--operation over the $t$--operation. As a matter of fact, since $F^{t}=F^{v}$ for each finitely generated nonzero ideal $F$,  by
\cite[Lemme 1]{[Q]}, we have:
\begin{equation*}
\begin{array}{rl}
D\mbox{ is integrally closed }\Leftrightarrow & \boldsymbol{c}_{D}(fg)^{v}=(
\boldsymbol{c}_{D}(f)\boldsymbol{c}_{D}(g))^{v}
\mbox { for all $0 \ne
f,g\in K[X]$} \\
\Leftrightarrow & \boldsymbol{c}_{D}(fg)^{t}=(\boldsymbol{c}_{D}(f)
\boldsymbol{c}_{D}(g))^{t}\mbox { for all $0 \ne f,g\in K[X]$}.
\end{array}
\end{equation*}
In other words, for a Gaussian-like characterization of P$v$MD's, $w$ can do what $t$ cannot do.
\medskip

 As  noted in the introduction, P$\ast $MD's were
introduced by Houston, Malik, and Mott in \cite{[HMM]} for a
finite type star operation $\ast $. Note that for any star
operation $\ast $, a $\ast $-invertible $\ast $-ideal is a
$v$-ideal (cf. \cite[Corollaire 1, page 21]{[J]},
\cite[Proposition 3.1]{[An]}). Now since in a P$\ast $MD every
star ideal of finite type is $\ast _{\!_{f}}$-invertible, and so
is a $v$-ideal of finite type, we conclude that in a P$\ast $MD,
where $\ast$ is a finite type star operation,  every $\ast $-ideal
is in fact a $t$-ideal.

\medskip

It follows immediately by definition that for two semistar
operations $\star_1$ and $ \star_2$ on $D$, if $\star_1 \leq
\star_2$ and if $D$ is a P$\star_1$MD, then $D$ is also  a
P$\star_2$MD. In particular, for each (semi)star operation $\star$
on $D$, we have that $D$ is a P$\star$MD implies that $D$ is also
a P$v$MD since  $\star \leq v$.  Given a P$v$MD $D$, one wonders
if there is a non-trivial (semi)star operation $\star$ of finite
type on $D$ such that $D$ is not a P$\star$MD. Fontana, Jara, and
Santos provided such an example in \cite[Example 3.4]{[FJS]}. The
next corollary shows the way to construct more examples.

\begin{corollary}
\label{Corollary C} Let $D$ be a P$v$MD, let $n \geq 1$, and let $
\boldsymbol{\mathcal{T}} := \{D_{S_{i}}\mid 1 \leq i \leq n\}$ be
a  finite   family of quotient rings of $D$ such that $D= \bigcap_{i=1}^{n}
D_{S_{i}}$. If some $D_{S_{i}}$ is  not a Pr\"ufer domain, then
$D$ is a P$v$MD with a stable (semi)star operation $\star$ of
finite type (e.g., $\star := \wedge_{ \boldsymbol{\mathcal{T}}}$)
such that $D$ is not a P$\star$MD.
\end{corollary}
\begin{proof} Let $\star :=\wedge _{\boldsymbol{\mathcal{T}}}$ be the stable
(semi)star operation induced by the family of overrings
$\boldsymbol{ \mathcal{T}}$, and suppose that $D_{S_{1}}$ is not a
Pr\"{u}fer domain. Then since $ \boldsymbol{\mathcal{T}}$ is
finite, and hence of finite character,  $\star $ is a (semi)star operation of
finite type 
\cite[Theorem 2 (4)]{[A]}.   By Corollary \ref{Corollary A3} ((iii)$\Rightarrow
$(ii)), $D$ is not a P$\star $MD. \end{proof}

By applying Corollary \ref{Corollary C}, the next corollary
provides further examples of P$v$MD's which are not P$\star $MD's
for some stable (semi)star operation $\star $ of finite type. For
the following statement, we fix a notation: given a $0\neq x\in
D$, we let $D_{x}$ be the quotient ring $ D_{S}$, where
$S:=\{x^{k}\mid k\geq 0\}.$

\begin{corollary}
\label{Corollary D} Let $D$ be a P$v$MD. Then the following hold:

\begin{enumerate}
\item[(a)] Suppose that $D$ has nonzero nonunits
$x_{1},x_{2},\ldots ,x_{n}$ with $(x_{1},x_{2},\ldots
,x_{n})^{v}=D$, $n\geq 2$, and $D_{x_{i}}$ is not a Pr\"{u}fer
domain for some $i$. Then there is a stable (semi)star operation
$\star $ of finite type on  $D$ such that $D$ is not a P$\star
$MD.

\item[(b)] Suppose that $M$ is a maximal ideal of $D$ with $D_{M}$ not a valuation
domain. If there is a  nonunit $x\in D\backslash M$, then $D$ has a
stable (semi)star operation $\star $ of finite type such that $D$ is not a P$
\star $MD.
\end{enumerate}
\end{corollary}

\begin{proof} (a) The proof hinges on the fact that $(x_{1},x_{2},\ldots
,x_{n})^{v}=D$ if and only if $D=\bigcap_{i=1}^{n}D_{x_{i}}$ \cite[Theorem 6]
{[Z2]}. Now the same procedure as in Corollary \ref{Corollary C} does the
rest of the job.

(b) Note that there is a $y\in M$ such that $(x,y)=D$. So, as in
(a), we have $D=D_{x}\cap D_{y}$, and $D_{x}$ is not a Pr\"{u}fer
domain since $ D_{x}\subseteq D_{M}$. \end{proof}

\medskip  Corollary~\ref{Corollary D}(b) can be applied for
instance to a non-quasilocal Krull domain of
dimension two. In particular, take $D:=K[X,Y]$, where $K$ is a field
and $X,Y$ are two indeterminates over $K$. Clearly $D$ is 
a non-Pr\"{u}fer P$v$MD. Let $M:=(X+1,Y)D$.
Observe that  $X \in D\backslash M$  is a nonunit,  $D_{M}$ is a Noetherian regular local 
domain of dimension two    (and thus it is not a valuation domain),   and that, for
instance, $(X,X+1)D=D$.

On the other hand, there do exist examples of non-Pr\"{u}fer
P$v$MD's $D$ such that for each pair of nonunits    $x,y \in D$   with
$((x,y)D)^{v}=D$,  we have that $D_{x}$ and $D_{y}$ are both
Pr\"{u}fer domains. For instance, take a  
two-dimensional  quasilocal Krull domain, e.g., $D:=K[\![X,Y]\!]$, where
$K$ is a field. (In this case, if $\alpha ,\beta \in D$ are
nonunits such that $ ((\alpha ,\beta )D)^{v}=D$, then $D_{\alpha
}$ and $D_{\beta }$ are Dedekind domains and $D=D_{\alpha }\cap
D_{\beta }$.)

\smallskip
 \smallskip

In the final part of    this  section, we examine the case of semistar
operations of the type $\star =\wedge _{\boldsymbol{\mathcal{T}}}$ without
assuming   finite character.

\begin{proposition}
\label{Prop-new1} Let $D$ be an integral domain with quotient
field $K$ and let $\star$ be the semistar operation induced by a
family $\boldsymbol{\mathcal{T}} $ of overrings of $D$, i.e.,
$\star =\wedge_{\boldsymbol{\mathcal{T}}}$. Consider the following
statements:

\begin{enumerate}
\item[(i)] \it
$\boldsymbol{c}_D(fg)^{{\star}}=(\boldsymbol{c}_D(f)\boldsymbol{c}_D(g))^{{
\star} }$    \ for all $0 \neq f, g\in K[X]$.

\item[(ii)] \textit{Each overring $T \in \boldsymbol{\mathcal{T}}$
is a Pr\"ufer domain. }

\item[(iii)] \it $ (FF^{-1})^\star =D^\star$  \  for all  $F
\in\boldsymbol{f}(D)$.
\end{enumerate}

\noindent Then {\rm
{(iii)}$\Rightarrow${(ii)}$\Leftrightarrow${(i)}}. \it Moreover,
if we assume that each $T \in \boldsymbol{\mathcal{T}}$ is a
quotient overring of $D$, then {\rm
{(iii)}$\Leftrightarrow${(ii)}$\Leftrightarrow${(i)}}.
\end{proposition}

\begin{proof} Since $\star =\wedge _{\boldsymbol{\mathcal{T}}}$, it is
easy to see that  $\boldsymbol{c}_{D}(h)^{{\star }}T=
\boldsymbol{c}_{D}(h)T=\boldsymbol{c}_{T}(h)$  for all $0 \ne h
\in K[X]$ and all $T\in \boldsymbol{\mathcal{T}}$.

(i)$\Rightarrow $(ii) Apply   the  Gauss-Gilmer-Tsang Theorem to each $
T\in \boldsymbol{\mathcal{T}}$.

(ii)$\Rightarrow $(i)  Since we are assuming that each overring
$T \in \boldsymbol{\mathcal{T}}$ is a Pr\"ufer domain,  we have $
\boldsymbol{c}_{D}(fg)^{\star }=\bigcap
\{\boldsymbol{c}_{T}(fg)\mid T\in
\boldsymbol{\mathcal{T}}\}=\bigcap
\{\boldsymbol{c}_{T}(f)\boldsymbol{c} _{T}(g)\mid T\in
\boldsymbol{\mathcal{T}}\}=(\boldsymbol{c}_{D}(f)
\boldsymbol{c}_{D}(g))^{\star }$  for all\ $0 \ne f,g\in K[X]$.

(iii)$\Rightarrow $(i) The proof is based on the Dedekind-Mertens
Lemma, and it is analogous to the proof of Theorem 1.1
((iii)$\Rightarrow $(i)).

Assume that each $T\in \boldsymbol{\mathcal{T}}$ is a quotient overring of $
D $.

(ii)$\Rightarrow $(iii) Since each $T\in \boldsymbol{\mathcal{T}}$
is a Pr\"{u}fer flat overring of $D$, every nonzero finitely
generated fractional ideal of $T$ is invertible and
$F^{-1}T=(FT)^{-1}$ for  all  
 $F\in \boldsymbol{f}(D)$. Therefore,
\begin{equation*}
\begin{array}{rl}
(FF^{-1})^{\star }= & \bigcap \{(FF^{-1})T\mid T\in \boldsymbol{\mathcal{T}}
\}=\bigcap \{FT(FT)^{-1}\mid T\in \boldsymbol{\mathcal{T}}\}= \\
= & \bigcap \{T\mid T\in \boldsymbol{\mathcal{T}}\}=D^{\star }\,.
\end{array}
\end{equation*}
\vskip -15pt\end{proof}

\begin{corollary}
\label{Cor-new2} Let $D$ be an integral domain with quotient field
$K$.  Set $E^b := \bigcap \{EV \mid V \mbox{ valuation overring of } D \}$ for all $E \in \F(D)$ (i.e., $E^b$ is the completion of the $D$--module $E$ in the sense of Zariski and Samuel \cite[Definition 1, page 347]{ZS}). If $ D $ is integrally closed, then
  $
\boldsymbol{c}_{D}(fg)^{b}=(\boldsymbol{c}_{D}(f)\boldsymbol{c} _{D}(g))^{b }$  \
for all\ $0 \neq f,g\in K[X]$.
\end{corollary}

\begin{proof}   Let $\boldsymbol{\mathcal{T}}$  be
the set of all valuation overrings of $D$. Clearly $b =\wedge
_{ \boldsymbol{\mathcal{T}}}$ and $b$ is a (semi)star operation on the integrally closed domain $D$ by Krull's Theorem \cite[Theorem 19.8]{[G]}. The statement  then  follows from
Proposition \ref{Prop-new1} ((ii)$ \Rightarrow $ (i)). \end{proof}

Recall that an integral domain $D$ is called an  \it essential
domain  \rm if there exists a set of prime ideals $\Delta $ of $D$
such that $D=\bigcap \{D_{P}\mid P\in \Delta \}$ and $D_{P}$ is a
valuation domain  for each $ P\in \Delta $. The set $\Delta $ is
called a \it  set of essential prime ideals for  \rm $D$. Every
P$v$MD is essential, and an essential domain having a set of
essential primes $\Delta $ of finite character (i.e, every nonzero
element of $D$ is a nonunit in only finitely many $D_{P}$, $P\in
\Delta $) is necessarily a P$v$MD \cite[pages 717-718]{Gr}.  In  \cite{Heinzer-Ohm}  
Heinzer and Ohm gave an example of an essential
domain which is not a P$v$MD. For more examples of non-P$v$MD essential domains consult Zafrullah \cite{z1988}.

\begin{corollary}
\label{Cor-new3}  Let $D$ be an integral
domain with quotient field $K$. Assume that $D$ is an essential
domain and let $ \Delta $ be a set of essential prime ideals for
$D$. Then   $\boldsymbol{c}_{D}(fg)^{{\star _{\Delta }}}=(
\boldsymbol{c}_{D}(f)\boldsymbol{c}_{D}(g))^{{\star _{\Delta }}}$
for all\ $ 0\neq f,g\in K[X]$.\  However, if $D$ is neither a
Pr\"{u}fer domain nor a quasilocal domain, then there exists a
stable (semi)star operation  $\star$ of finite type  on $D$,
defined by a family $ \boldsymbol{\mathcal{T}}$ of quotient
overrings of $D$ (i.e., $\star :=\wedge
_{\boldsymbol{\mathcal{T}}}$), such that $
\boldsymbol{c}_{D}(fg)^{{\star}}\neq
(\boldsymbol{c}_{D}(f)\boldsymbol{c} _{D}(g))^{{\star }}$   for
some\  $0\neq f,g\in K[X]$.  In particular, $D$ is not a P$\star
$MD.
\end{corollary}

 \begin{proof}  The first statement follows from Proposition
\ref{Prop-new1} ((ii)$\Rightarrow $(i)).  Now assume   that $D$ is
not quasilocal. If $M$ is a maximal ideal of $D$, then we can find
an element $x\in M$ such that $y:=1+x$ is not a unit in $D$.
Therefore, we have found two nonzero nonunits $x,y\in D$ such that
$ (x,y)=D$, and thus $D=D_{x}\cap D_{y}$. If $D$ is not a
Pr\"{u}fer domain, then there exists a maximal ideal $N$ of $D$
such that $D_{N}$ is not a valuation domain. Since at least one of
$x,y$ must avoid $N$ (i.e., $D_{x}\subset D_{N} $ or $D_{y}\subset
D_{N}$), then 
 $D_{x}$ or $D_{y}$ is not a Pr\"{u}fer
domain. Set $\boldsymbol{\mathcal{T}}:=\{D_{x},D_{y}\}$ and $\star
:=\wedge _{\boldsymbol{\mathcal{T}}}$. Clearly $\star$ is a stable
(semi)star operation of finite type on $D$. The conclusion follows
from Proposition \ref{Prop-new1} ((i)$\Rightarrow $(ii)) and
Theorem 1.1 ((i)$ \Rightarrow $(iii)).
\end{proof}


\section{Class Groups}

A somewhat interesting use of the
results of
Section 1 can be made, yet we need to introduce some terminology.
 While introducing the necessary terminology, we include some general
facts that either link this work with the literature or illuminate some
aspects of the theory of class groups.  This, apparently discursive,
treatment is also included to make a case for studying $\ast$--class groups
for star operations $\ast$ different from $t$. 
 
 Let $\text{%
Inv}^{t}(D)$ be the set of $t$--invertible $t$--ideals of an
integral
domain $D.$ Clearly $\text{Inv}^{t}(D)$ is an abelian group under $t$%
--multiplication and $\text{Inv}^{t}(D)$ contains $\text{Prin}%
(D)$, the set of nonzero principal fractional ideals of $D$, as a
subgroup.  The quotient-group
$\text{Cl}^{t}(D):=\text{Inv}^{t}(D)/\text{Prin}(D)$ is called the
\it    $t$--class group of $D$ \rm  (note that  it was
introduced in \cite{[B]} as \textquotedblleft the class group" of the arbitrary domain $D$%
). The $t$--class group has the interesting property that while it
is defined for any integral domain $D$, it is the divisor class
group of $D$ when
$D$ is a Krull domain and the ideal class group of $D$ when $D$ is a Pr\"{u}%
fer domain.  (Recall that in a Krull (resp., Pr\"ufer) domain $D$,
the nonzero fractional divisorial ideals $\F^v(D)$ (resp.,
nonzero finitely generated fractional ideals $\f(D)$) form an
abelian group under the $v$--operation (resp., $d$-operation,
i.e.,  usual product of ideals); the \it divisor class group \rm
(resp., \it ideal class group\rm) \it of $D$ \rm is the
quotient-group $\F^v(D)/ \Prin(D)$  (resp.,
$\f(D)/\Prin(D)$).  

Moreover, a P$v$MD $D$ is a GCD-domain if and only if $\text{Cl}%
^{t}(D)$ is trivial \cite[Corollary 1.5]{[BZ]}.  There are other results that indicate that $%
\text{Cl}^{t}(D)$ is intimately related with the divisibility properties of $%
D,$ see e.g., \cite{[B]}, \cite{[BZ]}, \cite{[Z1]}, \cite{[AZ]},
and \cite{[An2]}. For these
reasons, apparently, Halter-Koch \cite{[HK]} adapted the notion of the 
$t$--class group for monoids. In \cite{[An]}, D.F.  Anderson
surveyed the topic and introduced a generalization of
$\text{Cl}^{t}(D)$ by noting that if $\ast $ is a star operation
on $D$, then the set $\text{Inv}^{\ast }(D)$
of $\ast $--invertible $\ast $--ideals is an abelian group under $\ast $%
--multiplication and indeed $\text{Prin}(D)$ is a subgroup of $\text{Inv}%
^{\ast }(D)$. The quotient group $\text{Cl}^{\ast
}(D):=\text{Inv}^{\ast }(D)/\text{Prin}(D)$ is called the  \it
$\ast $--class group of $D$. \rm

It is also possible to define a $\star $--class group  for  a semistar operation $\star $ on an integral domain  $D$, but
the generalization is not straightforward.

Let $\star $ be a semistar operation on $D$. We say that $I\in  
\boldsymbol{\overline{F}}(D)$ is  \it quasi--$\star $--invertible  \rm (resp.,  \it $\star$--invertible\rm)  if   $\left( I(D^{\star }:I)\right) ^{\star }=D^{\star }$  
(resp., if $I\in \boldsymbol{{F}}(D)$ and \ $\left( I(D:I)\right) ^{\star
}=D^{\star })$. It is obvious that $\star $--invertible implies quasi--$
\star $--invertible, but the converse does not hold (even if
$\star $ is a stable semistar operation of finite type)
\cite[Example 2.9]{FP}. However, it is clear from the definition
that if $\star $ is a (semi)star
operation and if $I\in \boldsymbol{\overline{F}}(D)$ is quasi--$\star $
--invertible, then $I$ must belong to $\boldsymbol{{F}}(D)$, and so $I$ is $
\star $--invertible. It is not hard to prove that $I$ is quasi--$\star $
--invertible if and only if there exists an $H\in
\boldsymbol{\overline{F}}(D)$ such that $(IH)^{\star }=D^{\star }$
\cite[Lemma 2.10]{FP}.

In the following proposition, we recall some known facts on $\star $ 
--invertibility and quasi--$\star $--invertibility (cf. \cite[Propositions
2.15, 2.16, 2.18 and Corollary 2.17]{FP}).


\begin{proposition}
\label{qsi} Let $\star $ be a semistar operation on an integral domain $D$.

\begin{enumerate}
\item[(1)] Let $I\in \boldsymbol{\overline{F}}(D)$. Then $I$ is quasi--$%
\star _{\!_{f}}$--invertible if and only if $I$ and $(D^{\star }:I)$ are $%
\star _{\!_{f}}$--finite (hence, $\star $--finite) and $I$ is quasi--$\star $%
--invertible.
\end{enumerate}

For the following statements, we assume $I\in
\boldsymbol{{F}}(D)$.

\begin{enumerate}
\item[(2)] Let $I$ be quasi--$\star $--invertible. Then  $I$ is $\star $%
--invertible if and only if $(D:I)^{\star }=(D^{\star }:I)$\
(i.e., $\left( I^{-1}\right) ^{\star }=\left( I^{\star }\right)
^{-1}$).

\item[(3)] If $\star $ is a (semi)star operation, then $I$ is quasi--$\star $%
--invertible if and only if $I$ is $\star $--invertible.

\item[(4)] If $\star $ is stable and $I\in \boldsymbol{f}(D)$,
then $I$ is quasi--$\star $--invertible if and only if $I$ is
$\star $--invertible.

\item[(5)] $I$ is $\star _{\!_{f}}$--invertible if and only if $I$ is $%
\tilde{\star}$--invertible.
\end{enumerate}
\end{proposition}

If $\star $ is a semistar operation on an integral domain $D$,
then we can introduce a semistar multiplication   
\textquotedblleft \ $\times^{\!\star} $\ \textquotedblright\  (or, simply,  \textquotedblleft $\times $\textquotedblright\ , if there   is  no danger for ambiguity)  on the set \ $
\text{Inv}^{\star }(D):=\{I\in \boldsymbol{\overline{F}}(D)\mid I $
is $\star$--invertible   and $ I=I^{\star }\}$  \ by $I\times^{\!\star} 
J:=(IJ)^{\star }$. Note that $(\text{Inv}^{\star }(D),\times )$ is
not a group in general, because, for instance, it does not
have  an identity element (e.g., when $D^{\star }\in \boldsymbol{%
\overline{F}}(D)\smallsetminus \boldsymbol{F}(D)$).

On the other hand, $\text{QInv}^{\star }(D):=\{I\in \boldsymbol{\overline{F}}%
(D)\mid I\mbox{ is quasi--$\star$--invertible and }I=I^{\star
}\}$, with the semistar multiplication \textquotedblleft $\times
$\textquotedblright\ introduced above, is always an abelian group
with identity $D^{\star }$ and the unique inverse of $I\in
\text{QInv}^{\star }(D)$ is the $D$--module $(D^{\star
}:I)\in \boldsymbol{\overline{F}}(D)$ (it is not hard to prove that $%
(D^{\star }:I)$ belongs to $\text{QInv}^{\star }(D))$. This fact
also provides one of the motivations, in the semistar operation  setting, for
introducing and studying $\text{QInv}^{\star }(D)$ (and not just
$\text{Inv}^{\star }(D)$, as in the \textquotedblleft
classical\textquotedblright\ star operation case).
Moreover, it is not difficult to prove that $(\text{Inv}^{\star }(D),\times )%
\mbox{ is a group}$ if and only if $(D:D^{\star })\neq (0)$ \cite[page 657]%
{FP}.

In particular, if $\star $ is a (semi)star operation on $D$, then
as we have already observed, the    notions   of quasi--$\star
$--invertible and $\star $--invertible coincide. More precisely,
in this case, we have:  \begin{equation*} \QInv^{\star
}(D)=\Inv^{\star }(D)=\{I\in {\boldsymbol{F}} (D)\mid I\mbox{ is
}\star {\mbox{--invertible and }}I=I^{\star }\}\,.
\end{equation*}

Let  $\ast $ be a star operation on an integral domain $D$. It is
well known
that every $\ast $--invertible $\ast $--ideal is a $v$--invertible 
  $v$--ideal    (see e.g.,   \cite[Theorem 1.1(a)]{[Z3]}).  This property
has a semistar analog. Given a semistar operation $\star $ on $D$,
it is easy to see that the operation defined by
\begin{equation*}
E^{v(\star )}:=(D^{\star }:(D^{\star }:E))\,   \mbox{ \; for all }E\in 
\boldsymbol{\overline{F}}(D)\,
\end{equation*} 
is a semistar operation on $D$ and $\star \leq v(\star )$
\cite[Section 1.2.5]{P}.   Set $t(\star ):=v(\star )_{\!_{f}}$.
It is obvious that when $\star $ is a (semi)star operation, then
$v(\star )$ (resp., $t(\star )$) coincides with the  
\textquotedblleft classical \textquotedblright\   $v$--operation (resp., $t$ 
--operation) on $\boldsymbol{{F}}(D)$.

\begin{proposition}
\label{qsi-v} \cite[Corollary 2.12]{FP} Let $\star $ be a semistar
operation on an integral domain $D$ and let $I\in
\boldsymbol{\overline{F}}(D)$. If $I$ is quasi--$\star
$--invertible, then $I$ is quasi--$v(\star )$--invertible and
$I^{\star }=I^{v(\star )}$. In particular, ${\QInv}^{\star }(D)$
is a subgroup of ${\QInv}^{v(\star )}(D)$.
\end{proposition}

From the previous proposition (and its proof) and from Proposition
\ref{qsi}(1), we easily deduce that if $I$ is quasi--$\star
_{\!_{f}}$--invertible, then $I$ is quasi--$t(\star )$--invertible
and $I^{\star _{\!_{f}}}=I^{t(\star )}$  (cf. also \cite[Corollary
3.13]{P}).   In particular, $\text{QInv}^{\star _{\!_{f}}}(D)$ is
a subgroup of $\text{QInv}^{t(\star )}(D)$.
\medskip

At this point, it is clear that we can also define class groups in
the
semistar    operation setting.   We define   the \it  $\star $--qclass group of $D$ \rm  to be the abelian group $%
\text{QCl}^{\star }(D):=\text{QInv}^{\star }(D)/\text{Prin}(D)$,
and under
the assumption $(D:D^{\star })\neq (0)$, we define the \it   $\star $--class group of $ 
D$ \rm  to be the abelian group $\text{Cl}^{\star
}(D):=\text{Inv}^{\star }(D)/\text{Prin}(D)$.
Clearly if $\star $ is a (semi)star operation, then $\text{QCl}^{\star }(D)=%
\text{Cl}^{\star }(D)$. When $\star =d$ is the identity (semi)star
operation, then as in the classical case, we define  the \it Picard group of $D$ \rm  to be the abelian group $%
\text{Pic}(D):=\text{QCl}^{d}(D)=\text{Cl}^{d}(D):=\text{Inv}^{d}(D)/\text{%
Prin}(D)$, where $\text{Inv}^{d}(D)$ coincides with the group of the
\textquotedblleft usual\textquotedblright\ fractional invertible ideals of $%
D $.
Note that $\text{Inv}^{d}(D)$ is a subgroup of $\text{QInv}%
^{\star }(D)$  for each semistar operation $\star $ on $D$ (resp.,
of $\text{Inv}^{\star }(D)$  for each semistar operation $\star $
on $D$ such that $(D:D^{\star })\neq (0)$). Therefore, following the
classical
case considered in \cite{[B]}, we call the quotient-group $\text{QG}%
^{\star }(D):=\text{QCl}^{\star }(D)/\text{Pic}(D)$ (resp., $\text{G}^{\star
}(D):=\text{Cl}^{\star }(D)/\text{Pic}(D)$   for each semistar operation $%
\star $ on $D$ such that $(D:D^{\star })\neq (0)$)  the  \it local $%
\star $--qclass group \rm   (resp.,    \it  local $\star $--class
group\rm) \it  of $D$.   \rm  It is
straightforward that when $\star =t$, the group $\text{QG}^{t}(D)=\text{G}%
^{t}(D)$ coincides with the \textquotedblleft classical\textquotedblright\
local class group $\text{G}(D)$ \cite{[B]}.

From Proposition \ref{qsi-v}, we deduce
that for each semistar operation $\star $ on an integral domain $D$, we have $%
\text{QCl}^{\star }(D)\subseteq \text{QCl}^{v(\star )}(D)$, and
under the
assumption $(D:D^{\star })\neq (0)$, we have  $\text{Cl}^{\star }(D)\subseteq \text{Cl}%
^{v(\star )}(D)$. When $\star $ is a semistar operation of finite
type, then the previous inclusions can be replaced by
$\text{QCl}^{\star }(D)\subseteq
\text{QCl}^{t(\star )}(D)$ and  $\text{Cl}^{\star }(D)\subseteq \text{Cl}%
^{t(\star )}(D)$, respectively. Furthermore, if $\star $ is a
(semi)star operation (resp., a (semi)star operation of finite
type), then we have $\text{Pic}(D) \subseteq \text{Cl}^{\star
}(D)\subseteq \text{Cl}^{v}(D)$ (resp., $\text{Pic}(D) \subseteq
\text{Cl}^{\star }(D)\subseteq \text{Cl}^{ t}(D)$).

 \smallskip

For the purposes of the present section, from now on we will only consider
the \textquotedblleft classical\textquotedblright\ case of a star operation $%
\ast $. If $\ast $ is a star operation of finite type on $D$, then $\text{Cl}%
^{\ast }(D)$ provides an interesting generalization of the ($t$--)class
group, but

(a) $t$--invertibility being somewhat abundant and more closely
linked with divi\-si\-bi\-lity \cite{[Z3]}, the $t$--class group seems
to have more applications, especially in view of its similarity to
the divisor class group for Krull domains and the facility of the
$t$--operation with polynomial rings and with rings of fractions.

On the other hand,

(b) there are very few examples of  $\ast $--class groups that are
not $d$--class groups,  $t$--class groups, or $v$--class groups.

Of course we cannot do much about (a), but we
can use Corollary 1.8 to give examples of $\ast $--class groups such that $%
\text{Cl}^{\ast }(D)\subsetneq \text{Cl}^{t}(D)$ and of
$\text{Cl}^{\ast }(D)=\text{Cl}^{t}(D)$, where there is at least
one nonzero finitely generated  ideal $F$ of $D$ such that $F$ is
$t$--invertible, but not $\ast $--invertible.

\begin{proposition}
\label{Proposition E} Let $D$ be an integral domain in which every $t$--invertible
$t$--ideal is invertible. Then  $\Pic(D) = \Cl^{\ast }(D)= \Cl^{t}(D)$
for any star operation $\ast $ of finite type on $D$.
\end{proposition}
\begin{proof}
Indeed,  we have already observed that every $\ast $--invertible
$\ast $--ideal is a $t$--invertible $t$--ideal and every
invertible ideal is  a $\ast $--invertible $\ast $--ideal for any
star operation $\ast $ on $D$.   So   $\Inv(D) \subseteq
\Inv^\ast(D) \subseteq \Inv^t(D)$.   On the other hand, every
$t$--invertible $t$--ideal is invertible by hypothesis. Combining
these inclusions, we conclude that  $\Inv(D)= \Inv^\ast(D)
=\Inv^t(D)$. Hence   $\Pic(D) = \Cl^{\ast }(D)=\Cl^{t}(D)$   in
this case.
\end{proof}

Now are there any integral domains that satisfy the hypothesis of
Proposition \ref{Proposition E}\!~? Indeed, there are plenty. Recall from \cite
{[Z-preS]} that an integral domain $D$ is  \it a pre-Schreier domain \rm   if for $0\neq x,y,z\in
D$,  $x | yz$   implies that $x=rs$ \     for some $r,s \in D$    with   $r | y$     and   $s | z$.    Pre-Schreier domains are a generalization of GCD-domains
 (cf.   \cite{[C]} and \cite[Theorem 1]{Dribin}).   It was indicated in
\cite[Proposition 1.4]{[BZ]}   that if $D$ is a pre-Schreier
domain, then $ \Cl^{t}(D)=0$. Obviously if $D$ is a pre-Schreier
domain and $\ast $ is a star operation of finite type on $D$, then
we must have  $\Cl^{\ast }(D)= \Cl^{t}(D)=0$.   We are aiming at a
somewhat more general result:

\begin{corollary}
\label{Corollary F} Let $D$ be an integral domain.  If \  $\Cl^{t}(D_{M})=0$  for all maximal ide\-als $M$ of $D$ (e.g., if
$D$ is a locally GCD-domain \cite{[BZ]}), then $\Pic(D) =\Cl^{\ast
}(D)= \Cl^{t}(D)$ for any star operation $\ast$ of finite type on
$D$.

 \end{corollary}
 \begin{proof} Indeed, the local class group $\G(D)=
\Cl^{t}(D)/\Pic(D) = 0 $   if  $\G(D_M) =0$ (in particular,
if $\Cl^{t}(D_{M})= 0$) for each maximal ideal $M$ of $D$
\cite[Proposition 2.4]{[BZ]}.
\end{proof}

This leaves us with the task of providing an example of an integral domain $%
D $ such that  $\Cl^{\ast }(D)\subsetneq \Cl^{t}(D)$   (also see Example \ref{krull}).    For
this, we recall that an integral domain $D$ is a  \it
generalized GCD \rm (for short, \it G-GCD\rm) \it domain \rm
if for every  nonzero  finitely generated ideal  $F$
of $D$, we have that  $F^{v}$   is invertible \cite{[AA]}.
Moreover, $D$ is a G-GCD domain  if and only if $D$ is a  P$v$MD
which is a locally GCD domain  \cite [Corollary 3.4]{[Z-preS]}. So if we are
looking for a  P$v$MD   example of an integral domain $D$
with $\Cl^{\ast }(D)\subsetneq \Cl^{t}(D)$, then $D$ had better
not be a  G-GCD  domain.

\begin{proposition}
\label{Proposition G}  Let $D$ be a   P$v$MD  such that there are
nonzero nonunits $ x_{1},x_{2}, \dots,$ $ x_{n} \in D$ with $((x_{1},
x_{2},\dots ,x_{n})D)^{v}=D$. Suppose that for at least one  $j$,
$1\leq j\leq n$,  $D_{x_{j}}$ is not a  G-GCD
 domain. Let $\ast $ be the  stable star operation of finite
type   induced on $D$ by the finite family of overrings $\calT :=
\{D_{x_{i}} \mid  1\leq i \leq n \}$,   i.e., $\ast =
\wedge_{\calT}$, or  equivalently, $I^\ast := \bigcap_{i=1}^n
ID_{x_i}$  for each $I \in \F(D)$.   Then there  exists   a nonzero
finitely generated ideal   $F$  of $D$ such that   $F^{v}$   is
not a $\ast $--invertible $(\ast $--)ideal. Consequently, in $D$
we have $\Cl^{\ast }(D)\subsetneq \Cl^{t}(D)$.
\end{proposition}

\begin{proof} We note that given a  nonzero   finitely generated ideal  $F$   of $D$,  for  $F^{v}$
to be $\ast $--invertible it is essential that   $
(F^{-1}F^{v})D_{x_{i}}=D_{x_{i}}$   for each $i=1,2,\dots, n$.
Since for say $i=j$, $D_{x_{j}}$ is not a  G-GCD  domain, we
conclude that there is a finitely generated ideal  $H$   of
$ D_{x_{j}}$ such that  $H^{v_j}$   is not invertible in
$D_{x_{j}}$  (where $v_j$ denotes the $v$--operation on
$D_{x_j}$).    But as  $H$  is finitely generated in
$D_{x_{j}}$,  we can assume that  $H=FD_{x_{j}}$,   where
 $F$ is  a nonzero  finitely generated ideal of $D$.   But,
 as $D$ is a P$v$MD, $F$  is a $t$--invertible ideal of $D$;
thus  we have   $ H^{v_j}=F^{v}D_{x_{j}}$  \cite{[BZ]}.
Since  $H^{v_j}$  is not invertible, we  easily
conclude that  $F^{-1}F^{v}D_{x_{j}}\neq D_{x_{j}}$.   Hence
 $ (F^{-1}F^{v})^{\ast }\neq D$.   The ``consequently" part
is obvious.
\end{proof}

It seems important to also indicate the situations where there are
two distinct star operations, made from a general star operation
$\ast $, say $\mu (\ast )$ and $\nu (\ast )$, where $\mu (\ast
)\neq \nu (\ast)$, but  $\Cl^{\mu (\ast )}(D)=\Cl^{\nu (\ast
)}(D)$.   The constructions that we have in mind are the $\ast
_{\!_f}$ and the
 $\widetilde{\ast} \ (= \ast _{w})$  constructions from a
general star operation $\ast $ mentioned in the introduction. Now
$\ast _{\!_f}$ and  $\widetilde{\ast}$  are not always equal, but
as  a consequence of Proposition \ref{qsi}(5),
$\Inv^{\widetilde{\ast}}(D)= \Inv^{\ast _{\!_f}}(D)$.  Thus:

\begin{proposition}
\label{Proposition G1} Let $\ast $ be a star operation on an integral domain
$D$. Then  $\Cl^{\widetilde{\ast}}(D)= \Cl^{\ast _{\!_f}}(D)$.   In particular,  $
\Cl^{w}(D)=\Cl^{t}(D)$. \hfill $\Box$
\end{proposition}

Let $\F^v(D):= \{I \in \F(D) \mid I = I^v \}$ be the set of
divisorial fractional ideals of $D$. It is  well known    that
$\F^v(D)$  is a group under  the    $v$--multiplication,   $\times^{\!v}$,  if and only
if $D$ is completely integrally closed \cite[Theorem 34.3]{[G]}.
In this situation, the group $\F^v(D)/\Prin(D)$ is called
the \it divisor class group of $D$. \rm  The $t$--class group has
often been dubbed as a generalization of the divisor class group
because, as we remarked above,   $\Cl^{t}(D)$ is precisely the
divisor class group for a Krull domain $D$. But the $t$--class
group is in general  far away from the divisor class group  (when
defined).  For instance,  for a completely integrally closed
domain $D$, the divisor class group of $D$ is zero only if $I^{v}$
is principal for every $I \in \F(D)$. However, there do exist
completely integrally closed GCD-domains (in fact, rank-one
valuation domains) which contain nonprincipal proper $v$--ideals.
One example that comes to mind is a rank-one valuation domain $V$
with value group  $\mathbb Q$  the rationals. In this case, the
divisor class group of $V$ is nonzero (see \cite[Example
2.7]{[Z-gen]} for an elementary verification), but as $V$ is a GCD-domain,  $\Cl^{t}(V)$  is zero \cite[Example 1.2]{[BZ]}.

We next give  a complete verification of the valuation domain
example mentioned above. Let us start by noting that,  for a
valuation domain $D$ which is not a field, there are at most two
distinct star operations on $D$, the $v$--operation and the
operation $d=t$ \cite[Exercise 12, p. 431]{[G]}.  Therefore,
  for a valuation domain $D$, we have  $\Cl^{t}(D) = \Pic(D) = \{0\}$.
Let us also note that a valuation domain $D$  is completely
integrally closed if and only if $\dim (D) \leq 1 $   \cite[Theorem
17.5(3)]{[G]},  and in this case,   as observed above,
$\F^v(D)$ is a group under $v$--multiplication.  Thus in
this case, $\Cl^v(D)$ coincides with the divisor class group of
$D$.  Next, let $ \mathit{G}$ be the value group of the valuation
domain $D$ and let  $\omega: K^{\bullet }\rightarrow \mathit{G} $
be the valuation that gives rise to $D$   (where $ K^{\bullet} :=
K \setminus \{0\}$).
   When $\dim(D) =1$,
either $D$ is a DVR or $\mathit{G}$ is  a dense subspace of the real numbers $\mathbb{R}$
(cf. \cite[page 193]{[G]} or \cite[Chapitre 6, \S4, N. 5, Propositions 7 et 8]{Bourbaki}).

\begin{theorem}
\label{Theorem H} Let $D$ be a (one-dimensional) valuation domain with value
group $\mathit{G}\subseteq \mathbb{R}.$ Then
\begin{enumerate}
\item[\rm (1) \it] If $D$ is a DVR,
then \ $\Pic(D)= \Cl^{v}(D)= 0$.
\item[\rm (2) \it]  If $D$ is not a DVR, then \ $\Pic(D)=0$ and \ $\Cl^{v}(D)=\mathbb{R}/\mathit{G}
$.
\end{enumerate}
\end{theorem}

\it Proof. \rm Suppose that $D$ is not a DVR; so $\mathit{G}$ is
dense in $\mathbb{R}$. Define a map  $ \varphi: \F(D)\rightarrow
\mathbb{R}$ by $\varphi (I):=\sup \{\omega (x)\mid
I\subseteq xD $ for $x\in  K^{\bullet}\}$.  Note that
$\varphi $ is well-defined since

(a) \ $yD\subseteq I\subseteq xD$ for $x,y\in  K^{\bullet}$
implies $\omega (x)\leq \omega (y)$,

 (b) \ $\mathit{G}\subseteq \mathbb{R}$,  with $\mathbb{%
R}$\textbf{\ }complete, \ and

(c) \ $\varphi (xD)=\omega (x)$ for $x\in  K^{\bullet}$.

 Using these observations, we also note that $\varphi (I)=\sup \{\omega
(x)\mid  x \in K^\bullet \mbox { and  } \omega (x) \leq \omega
(i)$   for all  $0 \ne i\in I\}$.  Therefore,
 for  all  $0 \ne i\in I$,  we have $\omega (i)\geq
\varphi (I)$.   For the same reasons, if $I\subseteq xD$, then
$\omega (x)\leq \varphi (I)$. The proof  of Theorem \ref{Theorem
H}  then follows from the following four lemmas.

\begin{lemma}
\label{Lemma I} Let $D$ be as in Theorem \ref{Theorem H}.  Then
the following statements are equivalent for $I,J\in \F(D)$.
\begin{enumerate}

\item[\rm (i)\it]
$\varphi (I)=\varphi (J)$.

\item[\rm (ii)\it] $\{xD\mid I\subseteq xD \mbox{ for }  x\in
 K^{\bullet}\}=\{xD \mid J\subseteq xD \mbox{ for }  x\in  K^{\bullet}\}$.

 \item[\rm (iii)\it]$I^{v}=J^{v}$.

 \end{enumerate}
\end{lemma}

\begin{proof}
Clearly  (ii)$\Leftrightarrow $(iii)   and
(ii)$\Rightarrow $(i).   For  (i)$ \Rightarrow$(ii),
suppose that $\varphi(I)=\varphi(J)$, but there is an $ x\in
K^{\bullet}$ such that $I\subseteq xD$ and $J\nsubseteq xD$.
Then $ xD\subsetneq J$. So there is a $y\in J$ such that
$I\subseteq xD\subsetneq yD\subseteq J.$ But then as already
noted, we have $\varphi (J)\leq \omega (y)<\omega
(x)\leq \varphi (I)$, a contradiction.
\end{proof}

From the above lemma, it follows that $\varphi $ restricts to a
 (well-defined)   injective map
$\varphi:\F^v(D)\rightarrow \mathbb{R}$.

\begin{lemma}
\label{Lemma L}   Let $D$ and $\mathit{G}$ be as in Theorem
\ref{Theorem H}, and let $\varphi:  \F^v(D) \rightarrow \mathbb{R}$ be defined as above.  Then  $$\varphi
^{-1}(\mathit{G})=\Prin(D)\,.$$
\end{lemma}

\begin{proof}
Let $I\in \varphi ^{-1}(\mathit{G}).$ Then $\varphi (I)=\alpha \in \mathit{G}$; so there is an
$x\in  K^{\bullet}$ such that $\omega (x)=\alpha .$ Since $\omega (x)=\varphi
(xD)$, we have $\varphi (I)=\varphi (xD).$ By Lemma \ref{Lemma I}, we have  $%
I^{v}=xD$.   But as $I\in \F^v(D)$, we have  $I=I^{v}$.   Conversely,
suppose that $xD\in $ Prin($D$). Then $\varphi (xD)=\omega (x)\in \mathit{G}.$
\end{proof}

\begin{lemma}
\label{Lemma J}   Let $D$ and $\mathit{G}$ be as in Theorem
\ref{Theorem H}.  Then the map $\varphi: \F^v(D)\rightarrow \mathbb{R}$    defined above is surjective.
\end{lemma}

\begin{proof}
Let $\alpha \in \mathbb{R}$.  Define  $I_{\alpha} :=\bigcap \{xD \mid x\in  K^{\bullet}$ with $
\omega (x)\leq \alpha \}$.   Since $\mathit{G}$ is dense
in $\mathbb{R}$, there is a $ y\in  K^{\bullet}$ such that
$\omega (y)>\alpha $.  Therefore  $yD \subseteq xD$ for
each   $x\in  K^{\bullet}$  with $\omega (x)\leq \alpha$. This ensures,
in particular,  that $I_{\alpha} $ is nonzero, and so
$I_{\alpha} \in \F^v(D)$. In order to
establish the surjectivity, we show that $\varphi
(I_{\alpha} )=\alpha$.  Clearly $\varphi (I_{\alpha}
)\geq \alpha $ because $\mathit{G}$ is dense in $\mathbb{R}$.
Suppose that $\varphi (I_{\alpha} )=\beta >\alpha$. So
there is a $z\in  K^{\bullet}$ such that $I_{\alpha}
\subseteq zD$ and $\alpha < \omega (z)=\gamma \leq \beta $.
 Let $\gamma
^{\prime }\in \mathit{G}$ such that $\alpha <\gamma^{\prime }<\gamma  \leq  \beta $ and
let $\omega ( z')=\gamma ^{\prime }$ for some $ z'\in  K^{\bullet}$.  Then for any $
x\in  K^{\bullet}$ with $\omega (x)\leq \alpha$, we have $\omega (x)<\omega
( z')$, forcing $ z'  \in I_{\alpha} $.  Now from  $\gamma ^{\prime }<\gamma $,
we have $\omega ( z' )<\omega (z)$, which is equivalent to $zD\subsetneq
 z'D  \ ( \subseteq I_{\alpha})$,  a contradiction.
\end{proof}

\begin{lemma}
\label{Lemma K} The map $\varphi:  \F^v(D)\rightarrow \mathbb{R}$  is a group homomorphism.
\end{lemma}

\begin{proof}
Let $I,J\in \F^v(D)$. We show that $\varphi
((IJ)^{v})=\varphi (I)+\varphi (J).$ Let $I\subseteq xD$ and
$J\subseteq yD$  for $x,y\in  K^{\bullet}.$ Then
$IJ\subseteq xyD$, and hence $\gamma  :=  \varphi
((IJ)^{v})=\varphi (IJ)\geq \omega (xy)=\omega (x)+\omega (y).$ Thus $
\varphi ((IJ)^{v})\geq \varphi (I)+\varphi (J).$ Suppose that $\varphi
(I)+\varphi (J)<\varphi ((IJ)^{v}).$ Then there are $\alpha ,\beta \in \mathit{G}$
such that $\varphi (I)\leq \alpha ,$ $\varphi (J)\leq \beta ,$ and $\alpha
+\beta <\gamma .$ Choose $x,y\in  K^{\bullet}$ with $\omega (x)=\alpha $ and $
\omega (y)=\beta .$ Then $xD\subseteq I$ and $yD\subseteq J$; so $
xyD\subseteq IJ.$ Let $IJ\subseteq zD$ for $z\in  K^{\bullet}$. Then as $
xyD\subseteq zD$, we have $\omega (z)\leq \omega (xy)=\omega (x)+\omega
(y)=\alpha +\beta <\gamma$. So $\gamma =\varphi ((IJ)^{v})=\sup \{\omega
(z)\mid IJ\subseteq zD  \mbox{ and } z \in  K^\bullet\}\leq \omega (xy)=\alpha +\beta <\gamma $, a
contradiction.
\end{proof}

Given a rank-one valuation domain $D$,  if we assume in Theorem
\ref{Theorem H} that   $D$  is not a DVR and that   $\mathit{G}\neq \mathbb{R}$, then
$\Cl^{v}(D)$  (which coincides in this case with the divisor class
group of $D$)  is not zero, whereas the $t$--class group
$\Cl^{t}(D)  \ (= \Pic(D))$   is zero.
 Having shown that both the divisor class group and $t$--class group can
coexist without being equal, we conclude that the $t$--class group is not a
generalization of the divisor class group.


\section{$v$--class groups and valuation domains}

  In the previous section,  we have seen the divisor class group
as the $v$--class group  in the case of one-dimensional valuation domains (Theorem \ref{Theorem H}) and, more generally, for  completely  integrally closed domains  \cite[Theorems 17.5 (3) and  34.3]{[G]}.  But thanks to the generality of its
definition, the group  $\Cl^{v}(D)$  does not need $D$ to have any
special properties. In other words, the $v$--class group is
defined for any   integral   domain. Now let us note that the
$v$--operation being the coarsest    star   operation, the
$v$--class group (for $v \neq t$) has hitherto been neglected. 
  So, we do not have a lot of examples from the literature to offer.   However,   to show that the $v$--class group has a life of its   own and  some interesting properties, we study the $v$--class group of some integral domains of interest.

\smallskip

 Let us first start  with some relevant cases  where the $v$--class group is the same
as the $t$--class group.

 \begin{proposition} \label{cl0}  
 \rm (1) \it  Let $D$ be an integral domain
such that
 $F^{-1}$   is of finite type  for all  $F \in  \f(D)$. Then  $\Cl^{t}(D)=\Cl^{v}(D)$.

\begin{enumerate}

 \item[(2)]   Let $D$ be a  valuation domain  (in particular, $0 = \Pic(D) = \Cl^t(D)$).
 
{\begin{enumerate}

\item[(2a)] Assume that $D$ is  one-dimensional.  Then
$\Cl^{v}(D)=0$   if and only if either $D$ is a DVR or $D$ has
value group $\mathbb{R}$.   Moreover, $\Cl^v(D)$  is a divisible abelian
group and may have torsion elements.  However, $\Cl^v(D)$ is torsion-free if
the value group of  $D$ is
$\mathbb Q$.

\item[(2b)]  Assume that $D$ is an $n$-dimensional valuation domain,  $1 \leq n\leq
\infty $,  with maximal ideal $M$.
 If $M$ is principal,
then all nonzero fractional ideals of $D$ are divisorial. So
$d =v$, and thus  $ 0= \Pic(D)= \Cl^{t}(D)=\Cl^{v}(D) $.  
\end{enumerate}
}
\end{enumerate}
\end{proposition}

\begin{proof}
The proof of (1) is straightforward     since  by Proposition \ref{qsi} we
deduce that $I \in \F(D)$ is $t$--invertible if and only if $I$
and $I^{-1}$ are $t$--finite and $I$ is $v$--invertible.   The
main examples of domains for which this result holds are  Mori
domains, which include Noetherian and Krull domains as special
cases.

\rm For the proof  of  (2a)  consult Theorem \ref{Theorem H}.
The other statements are easy consequences of the first one since  $\mathbb{R}$  is an additive
divisible group and any quotient group of a divisible group is again
divisible.  It is easy to check that  $\mathbb{R}/\mathbb{Q}$  is
torsion-free. 

\rm (2b) is well known \cite[Example 12, page 431]{[G]}.    
\end{proof}

 If $D$ is an $n$-dimensional valuation domain, $1 \leq n\leq
\infty $,   with maximal ideal $M$ and $M$ is not principal, we get a completely different story.  
 We essentially devote the major part of the
remainder of this section to this case.  
To give a clear idea  of this situation,  we start with an
example.

\begin{example}
\label{Example M}  \sl Let $V$ be a nondiscrete rank-one valuation
domain with value group $\mathit{G}$, let $K$ be the quotient
field of $V$,   let $X$  be an indeterminate over $K$,  and
let $D  := V+XK[\![X]\!]$. Then  $D$ is a two-dimensional
valuation domain (with quotient field  $K(\!(X)\!)$); \ thus $ 0 = \Pic(D) = \Cl^t(D)$,
 and moreover,   $\Cl^{v}(D)=\mathbb{R}/\mathit{G}$.

\rm

 The fact that $D$ is a two-dimensional valuation domain
follows from  \cite[Proposition 18.2(3)]{[G]}.   By using a very general theory of the class groups on pullback constructions \cite{[FP0]}, we have $\Cl^{v}(D)\cong \Cl^v(V)$  (see also the following Theorem \ref{cl v/p}). As a matter of fact \cite[Corollary
2.7]{[FP0]}  ensures that, given a quasilocal integral domain $(T, M, k)$ and a proper subring $S$ of $k$, if $R$ is the integral domain arising from the following pullback of canonical homomorphisms
$$
\begin{CD}
R  @> >> S\\
@V  VV @ V VV\\
T@>   >> T/M = k\,,
\end{CD}
$$
then $\Cl^v(S) \cong \Cl^v(R)$.  The conclusion  then 
follows from Theorem \ref{Theorem H}(2).
\end{example}

Since the case of valuation domains is  rather   peculiar and relevant, it deserves  particular attention. In this case, in fact,  it is possible    to give  direct  proofs of special cases of more general results on $\ast$--class  groups concerning pullback constructions  (cf., in particular,  \cite{FG} and \cite{[FP0]})
by using  elementary direct methods that are  elegant and   simple   to handle.   The next goal is to show how,  in the previous Example \ref{Example M},  it is possible to avoid the use of \cite[Corollary
2.7]{[FP0]}.  
 
 \smallskip

Recall   that for  each fractional ideal $I$ of an integral domain $D$ with quotient field $K$,  we have  
$I^{v}=\bigcap \{zD \mid z \in K \mbox{ and } I\subseteq zD\}$.     Moreover, as observed in  \cite[page
432]{[Z3]},  $I^{v}\neq D$ if and only if there are $0 \neq a, b\in D$ such
that $I\subseteq (aD: bD)$ and $a \nmid b$.   Also,   recall from \cite{[DFA]} that an
integral domain $D$ is  an \it  \texttt{IP}--domain \rm if every proper integral $v$--ideal of $D$
is an intersection of integral principal ideals of $D$. 

Now let $I$ be a
proper integral ideal of a valuation domain $V$ with   quotient
field  $K$.   Then $I^{v}=$ $\bigcap \{zV \mid z \in K \mbox{ and } I\subseteq zV\}$. 
 Since $V$ is a valuation domain and $I$ an integral ideal of   $V$, we
have $I^{v}\subseteq V$,   and so
 $$
 \begin{array}{rl}
 I^ {v} =& \left(\bigcap \{zV \mid z \in V \mbox{ and } I\subseteq zV\}\right)\cap 
 \left( \bigcap \{zV \mid z\in K\setminus V \mbox{ and } I\subseteq zV\}\right) \\
 =& \bigcap \{zV \mid z \in V \mbox{ and } I\subseteq zV\}
 \end{array}
 $$
(note  that  for $z \in K$, $z\in K\setminus V$ is equivalent to $z^{-1}\in V$, and hence  $ \bigcap \{zV \mid z\in K\setminus V \mbox{ and } I\subseteq zV\} =V
$).
 So a valuation
domain $V$ is an \texttt{IP}--domain.  Indeed,   if $I$ is nonzero principal,  then $I=xV=I^{v}$. On the
other hand,   if $I$ is not principal,   then $I$ is not finitely generated. This 
leads to two cases: 
$$ \mbox{(a) $I^{v}=V$ \; \; or \; \; (b) $I^{v}=\bigcap \{zV \mid z \in V \mbox{ and } I\subseteq zV\subsetneq V\}$}.
$$
We now prepare to use the fact that if $V$ is a
valuation domain and $P$ is a nonmaximal prime ideal of   $V$,  then    $V/P$ is
a valuation domain that is not a field.

\begin{lemma} \label{v-quot} Let $V$ be a valuation domain with maximal ideal
$ M$, $P \subsetneq M$  a
prime ideal of  $V$, and $I $  an integral ideal of  $V$  with  $P \subsetneq I$.
Then $I^{v}/P=(I/P)^{v}$.    In particular,    $I/P$ is a $v$--ideal of $V/P$ if and
only if $I$ is a $v$--ideal of $V$.
\end{lemma}
\begin{proof}
 For $z\in V$,   we have $I\subseteq zV$ if and only
if $I/P\subseteq (zV)/P=(zV+P)/P$.    As\-su\-me  that $P \subsetneq I$
$\subseteq I^v \subsetneq V$.
 As above, 
 $I^{v}/P= \left( \bigcap \{zV \mid z \in V \mbox{ and } I\subseteq zV\subseteq V\}\right)/P $ = $
  \bigcap \{(zV)/P \mid z \in V \mbox{ and } I\subseteq zV\subseteq V\} $  =  $
  (I/P)^v$.
\end{proof}

To prove the main theorem of this  section,  we need to prove yet another lemma. 

\begin{lemma}  \label{v-inv-quot} Let $V$ be a valuation domain with maximal ideal $M$, $P \subsetneq M$  a
prime ideal of  $V$,  and $I$ an integral ideal of $V$ with $P\subsetneq I$.
Then $I/P$ is a $v$--invertible $v$--ideal of $V/P$ if and only if $I$ is a $v$--invertible $v$--ideal of $V$.
\end{lemma}

\begin{proof}  Let $I$ be a $v$--invertible ideal of $V.$ Since  $(II^{-1})^{v}=V$,   we
claim that $M\subseteq II^{-1}\subseteq V$.   Of these,   $II^{-1}\subseteq V$
always  holds;  so we concentrate on $M\subseteq II^{-1}$.    If $I$ is principal,
then $II^{-1}=V$,    and so trivially  $M\subseteq II^{-1}$. If on the other hand $I$ is
not finitely generated, then $II^{-1}=Q$   is   a prime ideal of $V$  \cite[Theorem 1]{[DDA]}. If $Q$
were nonmaximal,    then $Q$ must be divisorial (with $Q =(V: V_Q)$),    and so $V=(II^{-1})^{v}=Q 
\subsetneq V$, a contradiction. Thus $M=II^{-1}$, and in this case too,    $ 
M\subseteq II^{-1}$. 

 Next, since $P\subsetneq I$,   there exists an element  $
j\in I\backslash P.$ Let $J:=jI^{-1}$;  clearly    $J\subseteq V$. We claim that $
P\subsetneq J$.   For if not,  then  as we are working in a valuation domain $V$, we must have $
J=jI^{-1}\subseteq P$.  Multiplying both sides by $I$ and applying the  $v$--operation,   we get $jV = j (II^{-1})^{v}= (jII^{-1})^{v}\subseteq P$, because $P$ is a $v$-ideal
(being nonmaximal). This gives $j\in P$ a contradiction. 
Hence $P\subsetneq
J $. 

Now $(IJ)^{v}=jV$, and $I, J$, and $jV$ all properly contain $P$.   So    by
Lemma \ref{v-quot}, $(jV+P)/P= (jV)/P=(IJ)^{v}/P=(IJ/P)^{v}=((I/P)(J/P))^{v}$,  and this
establishes that $I/P$ is $v$-invertible in $V/P$.  Moreover, with an appeal
again to Lemma \ref{v-quot}, we conclude that if $I$ is a $v$--invertible $v$--ideal of $V$,  
then $I/P$ is a $v$--invertible $v$--ideal of $V/P$. 

Conversely,  if $I/P$ is a $
v$--invertible $v$--ideal of $V/P$,   then for some ideal $J$ of $V$ with $
P\subsetneq J$ we have   $((I/P)(J/P))^{v}= (xV)/P$,  where $x\in V\backslash P$.
From this fact,    it is easy to deduce that   $(IJ)^{v}=xV$, and so,   $I$ is a $v$--invertible $v$--ideal of $V$.
\end{proof}

\begin{theorem} \label{cl v/p} 
Let $V$ be a valuation domain with maximal ideal $M$ and proper 
quotient field $K$, and let $P\subsetneq M$ be a prime ideal of $V.$
Then $\Cl^{v}(V)\cong \Cl^{v}(V/P)$.
\end{theorem}

\begin{proof}  We first show that if there is a $v$--invertible $v$--ideal $I$ of  $V$, then    its class 
$[I]$ contains an integral ideal $J$ that properly contains $P$.  Let $\boldsymbol{\mathfrak{I}} \in
\Cl^{v}(V)$.   Then $\boldsymbol{\mathfrak{I}}=[I]$ for some $v$--invertible $v$--ideal $I$ of $V$.  Since 
$(II^{-1})^{v}=V$,  as in the proof of Lemma \ref{v-inv-quot},  we have $P\subsetneq II^{-1}$. 
Thus $ P \subsetneq jI$  for some $ j
\in I^{-1}$. 
%
%
 Let $J:=jI$. Define $\varphi (\boldsymbol{\mathfrak{I}}): =[J/P]$.  Note
that $[J/P]\in \Cl^{v}(J/P)$ by Lemmas \ref{v-quot} and \ref{v-inv-quot}.  

Since  $[J] = [I] $ in  $\Cl^v(V)$, it is
enough to study 
%
%
the case of integral $v$--invertible $v$--ideals $I\supsetneq
P,$ in which case $\varphi ([I])=[I/P].$ We first show that $\varphi $ is
well-defined. Let $A$ and $B$ be two $v$--invertible $v$--ideals of $V$ such
that $P\subsetneq A,B\subseteq V$ and $[A]=[B].$ Then $A=tB$ for some $0\neq
t\in K.$ Since $t\in V$ or  $t^{-1}\in V$, we  can assume that $t\in V$ (interchanging eventually $A$ with $B$).  Once
we assume that $t\in V$,  we find that $t\in V\backslash P$ because 
$P\subsetneq A.$ Thus $A/P=(tB)/P=((t+P)/P)(B/P)$ in $V/P$. So, $[A/P]=[B/P]$ in $
\Cl^{v}(V/P)$.  Thus $\varphi $ is well-defined. Next we show that $\varphi $
is injective. Suppose that $[A/P]=[B/P]$ in  $\Cl^{v}(V/P)$,  where $P\subsetneq
A,B\subseteq V$ are $v$--invertible $v$--ideals of $V$. Then as before we can
assume that $A/P=((t+P)/P)(B/P)$ for some $t\in V\backslash P$.  Thus $A/P=(tB)/P$,
which forces $A=tB$,   and hence $[A]=[B]$ in $\Cl^{v}(V)$. To show that $\varphi $ is surjective, let  $\boldsymbol{\mathfrak{J}}\in \Cl^{v}(V/P)$.
Then   $\boldsymbol{\mathfrak{J}}=[J]$ for some $v$--invertible integral $v$--ideal $J$ of $V/P$. By the
above comments and Lemmas  \ref{v-quot} and \ref{v-inv-quot}, $J=I/P$ for some $v$--invertible
integral $v$--ideal $I$ of $V$ such that $P\subsetneq I$. Thus $\varphi
([I])=[I/P]=[J]=\boldsymbol{\mathfrak{J}}$; so $\varphi $ is surjective. Finally, we show that $
\varphi $ is a group-homomorphism. Let  $[I], [I']\in \Cl^{v}(V)$, where  $P\subsetneq
I,I' \subseteq V$ are $v$--invertible $v$--ideals of $V$. Then $\varphi ([I]\cdot
\lbrack I'])=\varphi ([(II')^{v}])=[(II')^{v}/P]=[(II'/P)^{v}]$
 $=[((I/P)(I'/P))^{v}]$ $=[I/P] \cdot\lbrack I'/P]=\varphi ([I])\cdot \varphi
([I'])$. (Here we have used the fact that $P\subsetneq I,I'$ implies that $%
P\subsetneq II'$.) Thus $\varphi $ is an isomorphism. \end{proof}

The next statement,  which  has already appeared in Proposition \ref{cl0},  can be easily reobtained as a consequence of Theorem \ref{cl v/p}.

\begin{corollary} \label{v-princ}
 Let $V$ be a valuation domain with principal maximal ideal $M$. Then $\Cl^{v}(V)=0$.
\end{corollary}

\begin{proof}
Let $M=qV$ and set $P:=\bigcap_{n\geq1} q^{n}V$.  Then $P$ is a prime ideal of $V$, $P\subsetneq M$, 
and  there is no prime ideal between $P$ and $M$ \cite[Theorems 17.1 and 17.3]{[G]}. This makes $V/P$ a
discrete rank-one valuation domain, and so $\Cl^{v}(V/P)=0$. But then, by
Theorem \ref{cl v/p}, we have   $\Cl^{v}(V)\cong \Cl^{v}(V/P)=0$. 
\end{proof}

Let $V$ be a valuation domain such that the maximal ideal $M$ of $V$ is 
\textit{idempotent} (i.e., $M^{2}=M$) and \textit{branched} (i.e., has  an $M$
primary ideal different from $M$).  Recall that $M$ is idempotent if and only if $M$ is not finitely generated (i.e., not principal) and that $M$ is branched if and only if
there is a prime ideal $P\subsetneq M$ such that there is no prime ideal
between $P$ and $M$ \cite[Theorem 17.3]{[G]}. Let us call $P$  \textit{the prime
ideal directly below} $M$.  Indeed,  this makes $V/P$ a rank-one valuation
domain. If $M$ is  idempotent,  then it is easy to check that $M^{v}=V$ \cite[Corollary 3.1.3]{fhp}.

\begin{corollary} \label{r/g}
Let $V$ be a valuation domain with maximal ideal $M$ that is
branched and idempotent, let $P$ be the prime ideal directly below $M$, and let
 $\mathit{G}$ be the value group of $V/P$.
Then $\Cl^{v}(V)\cong \mathbb{R}/\mathit{G}$.
\end{corollary}
\begin{proof} By  Theorem  \ref{cl v/p}, we have $\Cl^{v}(V)\cong \Cl^{v}(V/P)$. Since $V/P$ is a
nondiscrete rank-one valuation domain, its  value group $\mathit{G}$  is isomorphic to a subgroup of  $\mathbb R$,   and thus   we have 
$\Cl^{v}(V/P)=\mathbb{R}/\mathit{G}$ (Theorem \ref{Theorem H} (2)). \end{proof}

Note that
Corollaries \ref{v-princ} and \ref{r/g}  let us compute  $\Cl^v(V)$ for any
finite-dimensional valuation domain  $V$.  Theorem  \ref{cl v/p} can also be used to give interesting statements relative
to the $D+M$ construction of Gilmer. 
For instance,
let $V'$ be a valuation domain that is expressible as  $K+M'$,  
where $K$ is a field and $M'$ the maximal ideal of $V'$. Also let $V$ be a
valuation domain  with quotient field $K$.   Then $D :=V+M'$
is a valuation domain such that $\Cl^{v}(D)\cong \Cl^{v}(V)$.

The reader may wonder about the nature of the elements of $\Cl^{v}(V)$ for the
valuation domain $V$ of Corollary \ref{r/g}. The clue comes from the proof of
Theorem \ref{cl v/p}  and the following result.

\begin{proposition}  Let $V$ be a valuation domain with quotient field $K$ and  with maximal ideal $M$ that
is idempotent and branched. Then for each nonprincipal $v$--invertible $v$--ideal $I$ of $V$,   there exist two $M\!$--primary ideals $Q$ and $Q_{1}$ of $V$  and two elements $0\neq x, y \in K$such that
$I=xQ$ and $I^{-1}=yQ_{1}$. Therefore \ $\Cl^{v}(V)=\{ [Q] \mid Q \mbox{  is an 
$M\!$--primary nonfinitely
generated $v$--invertible $v$--ideal of $V$}\} $ $ \cup $ $ \{[V]\}$.
\end{proposition}

\begin{proof} As in the proof of Theorem \ref{cl v/p}, if $I$ is a nonfinitely generated $v$--invertible integral $v$--ideal of  $V$,  then $II^{-1}=M.$ Since $M$ is
 branched,  we must have $I=xQ$ for some $0\neq x\in V$, where $Q$ is an 
 $M$--primary
ideal of $V$ \cite[Theorem 2]{[DDA]}. So $Q$ is a nonfinitely generated $v$--invertible $v$--ideal of $V$ 
 and  we have $
Q^{-1}= y^{-1}H$,  where $0\neq y\in V$ and $H$ is obviously a
nonfinitely generated $v$--invertible integral $v$--ideal of $V$. Another appeal to
\cite[Theorem 2]{[DDA]} gives that $H= zQ_{1}$, where $Q_{1}$ is $M$--primary and $
0\neq z\in V$. Thus $[I]=[Q]$ and $[I^{-1}]=[x^{-1}Q^{-1}]=[{(xy)}^{-1}H]=[z(xy)^{-1}Q_{1}]=[Q_{1}]$, as claimed in the
statement. The conclusion is now obvious once we note  that for   every
nonzero finitely generated (or, equivalently, principal) fractional ideal $I$ of $V$,  we have $[I]=[V]$, the identity element of $\Cl^{v}(V)$. \end{proof}

\smallskip

In order to extend   Example \ref{Example M},  we can start with the
value group  $\mathit{G'}$  of a valuation domain   $V'$ of any
dimension $ 1  \leq n\leq \infty $. If $\mathit{G}$ is a totally
ordered additive subgroup of $\mathbb R$,   we can construct the
lexicographic direct sum $ \Gamma :=\mathit{G'} \oplus \mathit{G}$.
The resulting group $\Gamma $ is a totally ordered abelian group upon
which we can construct,  using the
Krull-Kaplansky-Jaffard-Heinzer-Ohm Theorem (see for instance \cite[Corollary 18.5]{[G]} or \cite{mott}),   a valuation domain $D$ with a branched
maximal ideal $M$ and a prime ideal $P$ 
directly  below $M$
such that $D/P$ is a rank-one valuation domain with value group
$\mathit{G}$ \cite[page 223, Problem 20]{[G]} and $D_P$ is a valuation domain with value group
$\mathit{G'}$ \cite[proof of
Proposition (19.11)(3)]{[G]}.
 An explicit example of this type is the following.

\begin{example}  \label{Example N}\sl
Let $V$ be a rank-one valuation domain with value group $\mathit{G}$ 
and let $K$ be the quotient field of\  $V$. Let $V'$ be an
$n$--dimensional valuation domain, $1\leq n\leq \infty $,  with
value group $\mathit{G'}$ such that the residue field of \ $V'$ is
$K$ (note that this is possible by \cite[Corollary 18.5]{[G]}),
and let $\pi': V' \rightarrow K$ be the canonical projection.
Then the pullback $D:= \pi'^{-1}(V)$ is a   valuation domain   such that
 $D/P \cong V$ and $D_P= V'$, where $P:= \pi'^{-1}(0)$.  The value
group of $D$ is the lexicographic direct sum $\mathit{G'}\oplus
\mathit{G}$, $\dim(D) = n+1$ (resp., $\infty$) if $n \neq \infty$
(resp., $n = \infty$), and  $\Cl^v(D) = \mathbb  R /\mathit{G}$.
\rm

Arguing more or less as in Example \ref{Example M},     the properties listed  in the above 
statement follow from \cite[Proposition 18.2(3)]{[G]},
\cite[Chapitre 6, \S 10, N. 2, Lemme
2]{Bourbaki},  Theorem \ref{Theorem H}(2), and  Theorem \ref{cl v/p} (or, \cite[Corollary 2.7]{[FP0]}). \ec 
\end{example}

Note that in Example \ref{Example M},  if we set $S:=V\setminus
\{0\}$, then $D_{S}   = K[\![X]\!]$   is PID. So $D$ is an example of
a two-dimensional valuation domain such that $\Cl^{v}(D_{S})=0$,
while $\Cl^{v}(D)\neq 0$.    On the other hand,  we can construct
a valuation domain $D$  such that $\Cl^{v}(D)=0$,  but
$\Cl^{v}(D_{S})\neq 0$ for some ring of fractions $D_S$ of $D$.
For instance, using the techniques of Examples \ref{Example M} and
\ref{Example N}, if we construct a valuation domain $D$ having as
value group the lexicographic direct sum  $\mathbb Q\oplus \mathbb
Z$,  then $D$ is a two-dimensional valuation domain with the
height-one prime ideal  $P$    of $D$ such that   $D/P$   is a DVR (hence
the maximal ideal of $D$ is principal) and   $D_P$   is a rank-one
valuation domain with value group $\mathbb Q$.  Therefore
$\Cl^{v}(D)=0$  since every nonzero ideal of $D$ is divisorial   \cite[Exercise 12, page 431]{[G]},  
but  $\Cl^{v}(D_{P})=\mathbb{R} /\mathbb Q $.

\begin{example} \label{krull}
\bf (1)  \sl Let  $\mathit{G}$  be any abelian group. Then there is a quasilocal Krull domain $D$
with   $\Cl^v(D) = \mathit{G}$\ ($= \Cl^t(D)$  \cite[Corollary 44.3]{[G]})  and $\Pic(D) = 0$. \rm

  The first statement is due to L.G. Chouinard  \cite[Corollary 2]{Chouinard}.  It is obvious that $\Pic(D) = 0$
  since an invertible ideal in a quasilocal domain is principal \cite[page 72]{[G]}.

\bf (2) \sl Let $D  :=   K[X, Y]$, where $K$ is a field and $X, Y$
are two indeterminates over $K$.  Then there are infinitely many
distinct star operations on $D$   (more precisely, the cardinal
number  is  $2^\alpha$, where $\alpha := \max \{|K|, \aleph_0\}$  \cite[page 1639]{AA2}),  and  $\Cl^\ast(D) = 0$    for all star
operations $\ast$ on $D$. \rm

 Since $D$ is a UFD,
thus  $\Cl^v(D) = 0$ \cite[Corollary 44.5]{[G]}. The conclusion
follows from the fact that $\Cl^\ast(D) \subseteq \Cl^v(D)$  for all   star  operations $\ast$ on $D$.

\end{example}

\begin{remark} \label{H}  \rm
Recall from \cite[p. 651]{FP} that an integral  domain   $D$ with a semistar
operation $\star$ is an \it  \texttt{H}$(\star)$--domain \rm  if for each nonzero
integral ideal $I$ of $D$ such that $I^\star = D^\star$,  there
exists a nonzero finitely generated ideal $J$  with $J \subseteq I$,  such that $J^\star = D^\star$ (i.e., $I$ is $\star_{_{\!f}}$--finite). 
When $\star =v$, the \texttt{H}$(v)$--domains coincide with the \texttt{H}--domains introduced by Glaz and Vasconcelos \cite[Remark 2.2 (c)]{Glaz/Vasconcelos:1977}.

It is obvious that  every integral  domain    is an  \texttt{H}$(\star_{_{\!f}})$--domain, so the notion of  \texttt{H}$(\star)$--domain  becomes interesting   only when $\star$ is not of finite type.

Clearly a $\star$--Noetherian domain (i.e.,  an integral domain having the ascending chain condition on the quasi--$\star$--ideals \cite[Section 3]{EFP})  is an \texttt{H}$(\star)$--domain \cite[Lemma 3.3]{EFP}, thus we obtain in particular  that Mori domains (or $v$-Noetherian domains; e.g. Noetherian and Krull domains) are \texttt{H}--domains. 
Houston and Zafrullah \cite[Proposition 2.4]{HZ88} proved, more generally,  that each \it \texttt{TV}-domain \rm  (i.e., an integral domain such that the $t$-operation coincides with the $v$--operation)  is an \texttt{H}--domain. Conversely, a general class of \texttt{H}--domains which are not  \texttt{TV}-domain was also given in \cite{HZ88}.  

It was shown in \cite[Proposition 4.2]{[Z3]} that an integral domain is an \texttt{H}--domain if and only if every $v$--invertible ideal is $t$--invertible.  This statement can be easily generalized to the arbitrary star operation setting.  

\sl Let $\ast$ be a star operation
on an integral domain $D$. Then the following conditions are
equivalent:
\begin{enumerate}
\item[(i)] $D$ is an  \texttt{H}$(\ast)$--domain (resp., an   \texttt{H}--domain).

\item[(ii)] Each $\ast_{_{\!f}}$--maximal ideal of $D$ is a
$\ast$--ideal of $D$ (resp., each maximal  $t$--ideal is a $v$--ideal). 

\item[(iii)]  For each $I \in\boldsymbol{{F}}(D)$,\ $I$ is $\ast$--invertible
if and  only if $I$ is $\ast_{\!_f}$--invertible (resp.,  $I$ is $v$--invertible
if and  only if $I$ is $t$--invertible).

\item[(iv)]  $\Cl^{\ast_{_{\!f}}}(D) = \Cl^{\ast}(D)$ (resp.,  $\Cl^{t}(D) = \Cl^{v}(D)$).

\end{enumerate} 
\rm The equivalences (i)$\Leftrightarrow$(ii)$\Leftrightarrow$(iii) follow from \cite[Proposition 11]{FP2}. 
(iii)$\Rightarrow$(iv)  Since a $\ast$--ideal is trivially also   a $\ast_{_{\!f}}$--ideal,  thus   under the assumption (iii),  $\Inv^{\ast}(D) \subseteq \Inv^{\ast_{_{\!f}}}(D)$. Therefore $\Cl^{\ast}(D) \subseteq \Cl^{\ast_{_{\!f}}}(D)$. Since the reverse inclusion holds in general, we have $\Cl^{\ast}(D) =\Cl^{\ast_{_{\!f}}}(D)$.  (iv)$\Rightarrow$(iii)   In this situation, $ \Inv^{\ast_{_{\!f}}}(D)= \Inv^{\ast}(D)$,  i.e., 
for each $I \in\boldsymbol{{F}}(D)$,\ $I$ is a $\ast$--invertible $\ast$--ideal
if and  only if $I$ is a $\ast_{\!_f}$--invertible $\ast_{\!_f}$--ideal.   Recall that if an ideal $I$ is  $\ast$--invertible   (resp., $\ast_{\!_f}$--invertible),   then $I^\ast = I^v$ (resp.,  $I^{\ast_{\!_f}} = I^t$) (Proposition \ref{qsi-v}).  Therefore, from the previous considerations we deduce that if $I$ is   
 $\ast$--invertible,    then it is 
 $\ast_{\!_f}$--invertible. The converse is trivial.
\end{remark}

 From the previous   remark,   we deduce immediately the following two corollaries.  The first one generalizes Corollary \ref{v-princ}  (cf. also Proposition \ref{cl0} (2b)).

\begin{corollary}  Let  $D$  be a quasilocal integral
domain with principal maximal ideal  $M$. Then $\Cl^\ast(D) = 0 $ for all star operations $\ast$ on $D$. 
\end{corollary}

\begin{proof} By Remark \ref{H} ((ii)$\Rightarrow$(i)),    since $M$ is principal,  $D$ is an  \texttt{H}--domain, and since $D$    is quasi--local,  $\Inv(D) =\Prin(D)$, and so $\Cl^t(D) = 0$.  Therefore $\Cl^\ast(D)=\Cl^t(D )=\Cl^v(D)=0$ for all star operations $\ast$.   
\end{proof} 

The next corollary generalizes Example \ref{krull} (2).

\begin{corollary}  If $D$ is a pre-Schreier domain and an  \texttt{H}--domain   (e.g., if $D$ is a UFD), then  
$\Cl^\ast(D) = 0 $ for    all star operations $\ast$ on $D$.
\end{corollary}

\begin{proof} We have already observed above that $\Cl^t(D) =0$ for a  pre-Schreier domain $D$ \cite[Proposition 1.4]{[BZ]}. The parenthetical statement follows from the fact that a GCD-domain is a pre-Schreier domain, and  a UFD  is a GCD-domain and a Krull domain.
\end{proof}

We have seen that if $V$ is a valuation domain  with principal maximal ideal  $M$,
then $\Cl^{v}(V)=0$  (Proposition \ref{cl0} (2b)),   and we have found out that if the maximal ideal 
$M$ is idempotent and  branched,   then $\Cl^{v}(V)$ is isomorphic to $\mathbb{R}/
\mathit{G}$,  where $\mathit{G}$ is the value group  of a certain
nondiscrete rank-one valuation domain  (Corollary \ref{r/g}). We also know that if the maximal
ideal $M$ of $V$ is not idempotent then $M$ is principal. This leaves us with
the case of $V$ with $M$ unbranched (and thus  necessarily idempotent  \cite[Theorem 17.3 (b)]{[G]}). At present we know very
little about this case, but an example in \cite[Example 1]{AA84} gives a
valuation domain $V$ with unbranched maximal ideal $M$ such that $V$ affords
a nonfinitely generated $v$--ideal  $I$ with $II^{-1}=M$.    Consequently,    a
valuation domain with unbranched maximal ideal may have nonzero $v$--class
group.

\smallskip \smallskip

We end the paper with some observations in the   not   necessarily quasilocal setting, but  related to the valuation domain case. Recall that  Bouvier \cite[Proposition 2]{[B]}  proved that $D$ is a GCD-domain if and
only if $D$ is a P$v$MD  and $\Cl^{t}(D)=0$ and, moreover,  in \cite[Proposition 1.4]{[BZ]}   Bouvier and Zafrullah have shown that if $D$ is a pre-Schreier domain, then $\Cl^{t}(D)=0$.  An interested reader
may want to know if results similar to these  hold for $\Cl^{v}(D)$. It
appears that one answer may suffice for both questions. Because $\Cl^{t}(D)\subseteq \Cl^{v}(D)$,  
 and so
if $D$ is a P$v$MD with  $\Cl^{v}(D)=0$,   then $D$ is a GCD-domain with a slight
difference. The difference is that not every GCD-domain $D$ has $\Cl^{v}(D)=0$.   
One example comes from Theorem \ref{Theorem H} and another slightly more general
example follows from Corollary \ref{r/g}. \ Indeed, as every GCD-domain is also
\it a Schreier domain \rm (i.e., an integrally closed domain in which every element is primal in Cohn's sense \cite{[C]}),   and hence a pre-Schreier domain \cite{Dribin}, and as a P$v$MD is pre-Schreier (or,  Schreier) if and only if it is a GCD-domain \cite[Corollary 1.5]{[BZ]},  we conclude that for a pre-Schreier domain $D
$ it is not necessary that $\Cl^{v}(D)=0$ (for an explicit example of this type cf. for instance \cite{[C]}).  We can however make a somewhat
more general statement in this connection.  For this,   recall that an integral
domain $D$ is a \it $v$--domain \rm  if every nonzero finitely generated ideal of $D$
is $v$--invertible.

\begin{proposition} \label{gcd}
Let $D$ be a $v$--domain  (e.g., a completely integrally closed integral domain \cite[Theorem 34.3]{[G]}).    If $\Cl^{v}(D)=0$,   then $D$ is a GCD-domain, but not
conversely.
\end{proposition}
\begin{proof} Let $I$ be a nonzero finitely generated ideal of $D$. Then $
(II^{-1})^{v}=D$.  So $I^{-1}$, being a $v$--invertible $v$--ideal, must be
principal because $\Cl^{v}(D)=0$. So for every nonzero finitely generated
ideal $I$    of $D$,    we have that $I^{v}=(I^{-1})^{-1}$ is principal, which  makes $D$ a
GCD-domain. That the converse is not true follows from comments prior to
this proposition.
\end{proof}

\begin{remark} \rm 
(1)   The previous, not very striking, statement gives us some
interesting candidates for the study of  the $v$--class group.

\begin{enumerate}
\item[(a)]  Nagata (in \cite{Nagata 1} and \cite{Nagata 2})  gave an example of a completely
integrally closed one-dimensional quasilocal integral domain $D$ that is not
a valuation domain. Obviously $D$ is not a GCD-domain (since,  as recalled above, a GCD-domain is a particular P$v$MD).   So   by Proposition \ref{gcd},  $\Cl^{v}(D)\neq 0$.
It   would  be of interest to find $\Cl^{v}(D)$ in this case.

\item[(b)]  Let $k$  be a field, let $Y, Z, X_1, X_2,\dots, X_n, \dots$  be indeterminates,  set $K:=k(Y, Z, X_1, X_2, \dots, X_n, \dots)$ and $R:=k(X_1, X_2,\dots, X_n, \dots)[Y,Z]_{(Y,Z)}$. Inside the field $K$, Heinzer and Ohm \cite[Section 2]{Heinzer-Ohm} and  
Mott and Zafrullah \cite[Example 2.1]{MZ}  consider a domain $D$ which
is not a P$v$MD. This domain $D$ is the intersection of the regular local ring $R$  with a denumerable family of  discrete valuation domains of rank   one  in $K$,  and so $D$ is 
 completely integrally closed. Not being a P$v$MD makes $D$  non-GCD,    and
so $\Cl^{v}(D)\neq 0$.  Again, it would be of interest to know $\Cl^{v}(D)$. 
\end{enumerate}

 The main question in both cases is: must $\Cl^{v}(D)$ be a homomorphic image
of $(\mathbb{R}, +)$ as we saw in the valuation domain  cases\!~?

(2)  Our study of $v$--class groups appears to raise a lot of other questions. We
mention some of those questions here.

\begin{enumerate}

\item[$\bullet$]  What is $\Cl^{v}(D[X])$ \ec in terms of  $\Cl^{v}(D)$\!~?  

\item[$\bullet$]   Halter-Koch in \cite[pages 167--185]{[HK]} talks
about valuation monoids. Consider ${\mathcal{S}}$ a valuation monoid
under addition. Defining the $v$--operation  on ${\mathcal{S}}$ and $\Cl^{v}({\mathcal{S}})$ in the obvious
\ec fashion,   find analogs of Theorems \ref{Theorem H} and \ref{cl v/p}.   Also,  study $\Cl^{v}(D[{\mathcal{S}}])$
when $\Cl^{v}(D)$  is known.     In particular,  find $\Cl^{v}(D[\mathbb{Q}^{+}])$ in terms of $\Cl^{v}(D)$. 

\end{enumerate}

\end{remark}

\ec

\medskip

\end{document}